\documentclass{amsart}

\usepackage{graphicx}
\usepackage[mathcal]{eucal}
\usepackage{epsfig}
\usepackage{graphicx}
\usepackage{pinlabel}
\usepackage{amscd}
\usepackage{algorithm}
\usepackage{algorithmic}

\newtheorem{theorem}{Theorem}[section]
\newtheorem{lemma}[theorem]{Lemma}

\theoremstyle{definition}
\newtheorem{definition}[theorem]{Definition}

\theoremstyle{remark}
\newtheorem{remark}[theorem]{Remark}

\numberwithin{equation}{section}

\begin{document}

\title{The Unified Surface Ricci Flow}

\author{Min Zhang}
\address{ Min Zhang, Department of Computer Science,}
\address{State University of New York at Stony Brook, NY 11794, USA}
\email{mzhang@cs.sunysb.edu}

\author{Ren Guo}
\address{ Ren Guo, Department of Mathematics,}
\address{ Oregon State University, OR 97331, USA}
\email{ guore@math.oregonstate.edu}

\author{Wei Zeng}
\address{Wei Zeng, School of Computing and Information Sciences,}
\address{Florida International University,FL 33199, USA}
\email{ wzeng@cs.fiu.edu}

\author{Feng Luo}
\address{ Feng Luo, Department of Mathematics, }
\address{Rutgers University, NJ 08854, USA}
\email{fluo@math.rutgers.edu}

\author{Shing-Tung Yau}
\address{ Shing-Tung Yau, Deaprment of Mathematics, }
\address{Harvard University, MA 02138, USA}
\email{ yau@math.harvard.edu}

\author{Xianfeng Gu}
\address{ Xianfeng Gu, Department of Computer Science, }
\address{State University of New York at Stony Brook, NY 11794, USA}
\email{gu@cs.sunysb.edu}

\maketitle

\begin{abstract}
Ricci flow deforms the Riemannian metric proportionally to the curvature, such that the curvature evolves according to a heat diffusion process and eventually becomes constant everywhere. Ricci flow has demonstrated its great potential by solving various problems in many fields, which can be hardly handled by alternative methods so far.

This work introduces the unified theoretic framework for discrete
Surface Ricci Flow, including all the common schemes: Tangential
Circle Packing, Thurston's Circle Packing, Inversive Distance Circle
Packing and Discrete Yamabe Flow. Furthermore, this work also
introduces a novel schemes, Virtual Radius Circle Packing and the
Mixed Type schemes, under the unified framework. This work gives
explicit geometric interpretation to the discrete Ricci energies for
all the schemes with all back ground geometries, and the
corresponding Hessian matrices.

The unified frame work deepens our understanding to the the discrete surface Ricci flow theory, and has inspired us to
discover the new schemes, improved the flexibility and robustness of the algorithms, greatly simplified the implementation and improved the efficiency.
\end{abstract}

\section{Introduction}
\label{sec:introduction}

Ricci flow was introduced by Hamilton for the purpose of studying
low dimensional topology. Ricci flow deforms the Riemannian metric
proportional to the curvature, such that the curvature evolves
according to a heat diffusion process, and eventually becomes
constant everywhere. In pure theory field, Ricci flow has been used
for the proof of Poincar\'{e}'s conjecture. In engineering fields,
surface Ricci flow has been broadly applied for tackling many
important problems, such as parameterization in graphics
\cite{Jin08TVCGRicci}, deformable surface registration in vision
\cite{TPAMI10Ricci}, manifold spline construction in geometric
modeling \cite{DBLP:journals/cad/GuHJLQY08} and cancer detection in
medical imaging \cite{TVCG10Colon}. More applications in engineering and medicine fields can be found in \cite{ricci2013}.

Suppose $(S,\mathbf{g})$ is a metric surface, according to the
Gauss-Bonnet theorem, the total Gaussian curvature $\int_S K dA_{\mathbf{g}}$ equals
to $2\pi \chi(S)$, where $K$ is the Gaussian curvature, $\chi(S)$
the Euler characteristics of $S$. Ricci flow deforms the Riemannian
metric conformally, namely, $\mathbf{g}(t)=e^{2u(t)}\mathbf{g}(0)$,
where $u(t): S\to \mathbb{R}$ is the conformal factor. The
normalized Ricci flow can be written as
\begin{equation}
    \frac{du(t)}{dt} = \frac{2\pi\chi(S)}{A(0)}-K(t).
    \label{eqn:smooth_ricci_flow}
\end{equation}
where $A(0)$ is the initial surface area. Hamilton
\cite{Hamilton_1988} and chow \cite{Chow_1991} proved the
convergence of surface Ricci flow. Surface Ricci flow is the
negative gradient flow of the Ricci energy. It is a powerful tool for
designing Riemannian metrics using prescribed curvatures, which has
great potential for many applications in engineering fields. Surface
Ricci flow implies the celebrated surface uniformization theorem as
shown in Fig.\ref{fig:uniformization_closed_surfaces}. For surfaces
with boundaries, uniformization theorem still holds as illustrated
in Fig.\ref{fig:uniformization_open_surfaces}, where surfaces are
conformally mapped to the circle domains on surfaces with constant
curvatures.


\begin{figure}[t!]
\begin{center}
\begin{tabular}{ccc}
\includegraphics[width=1.0in]{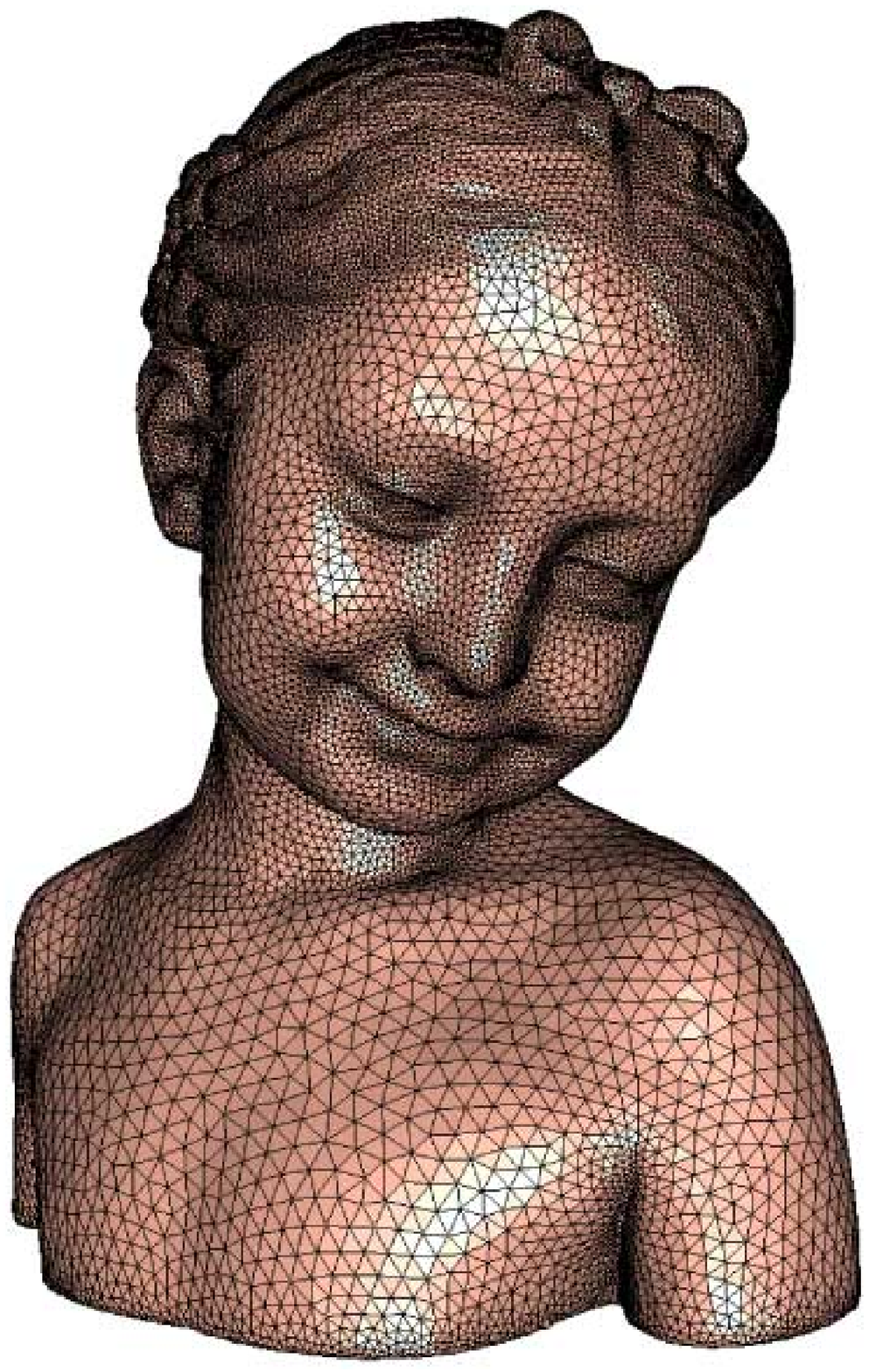}&
\includegraphics[width=1.0in]{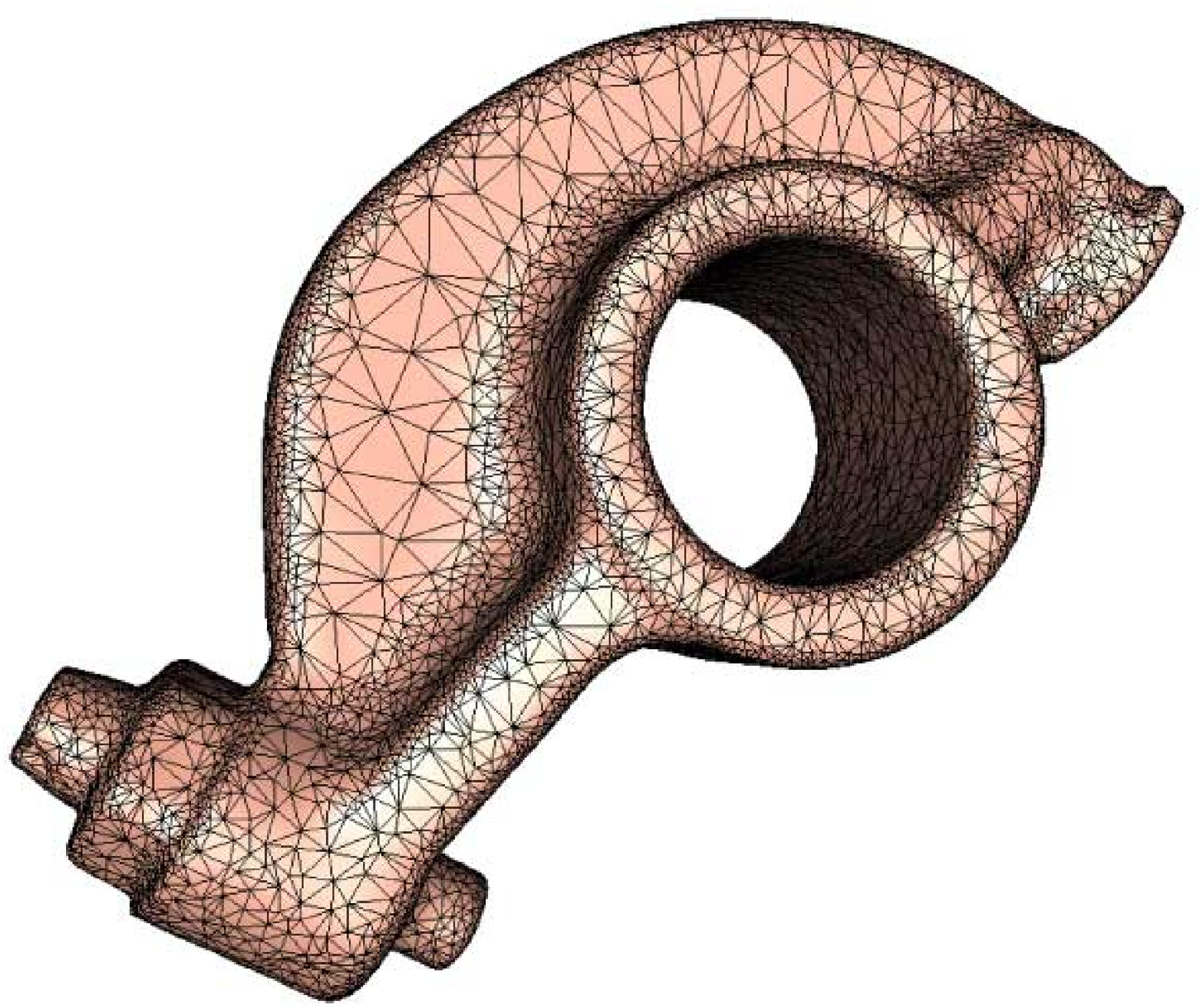}&
\includegraphics[width=1.0in]{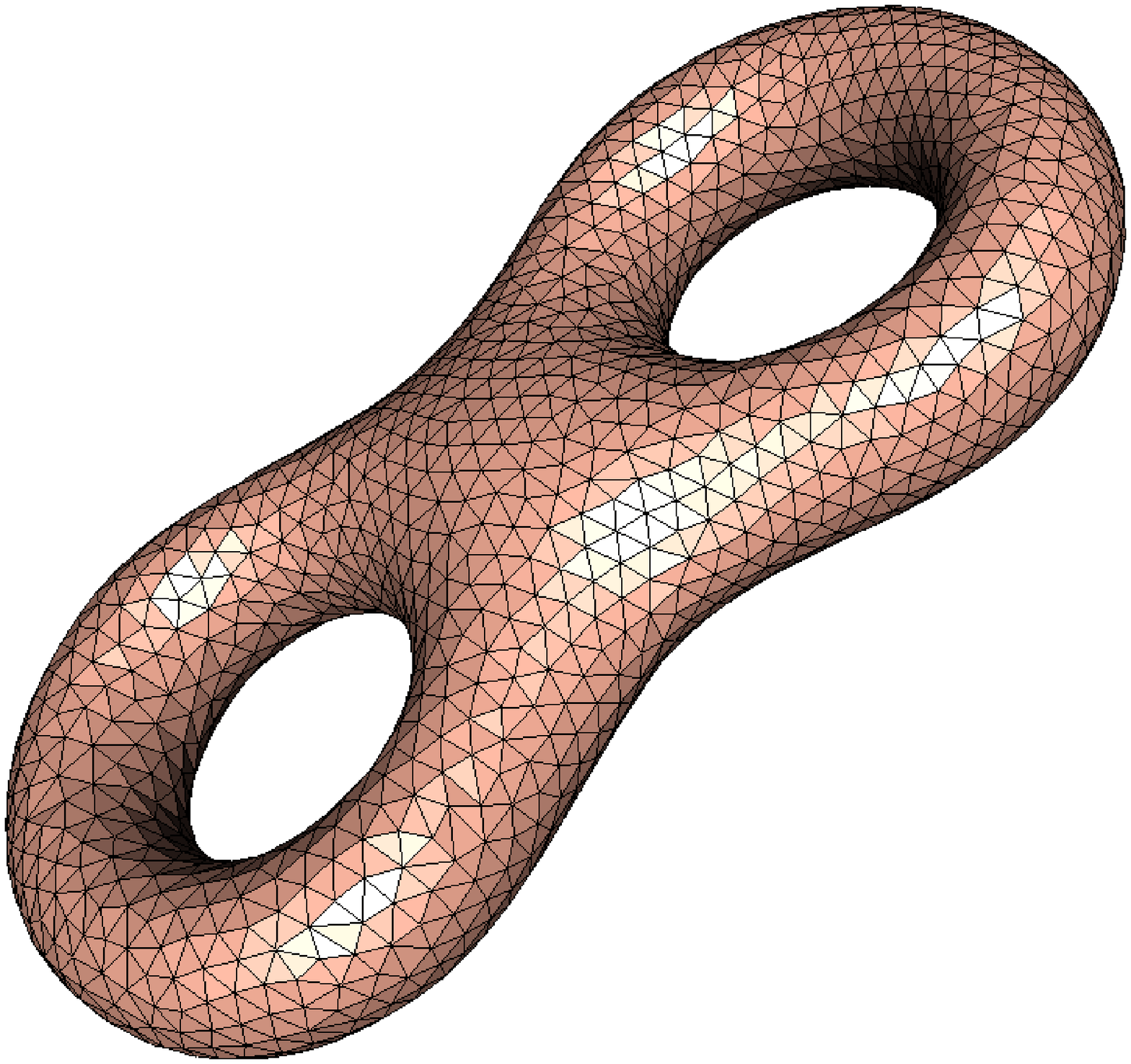}\\
\includegraphics[width=1.0in]{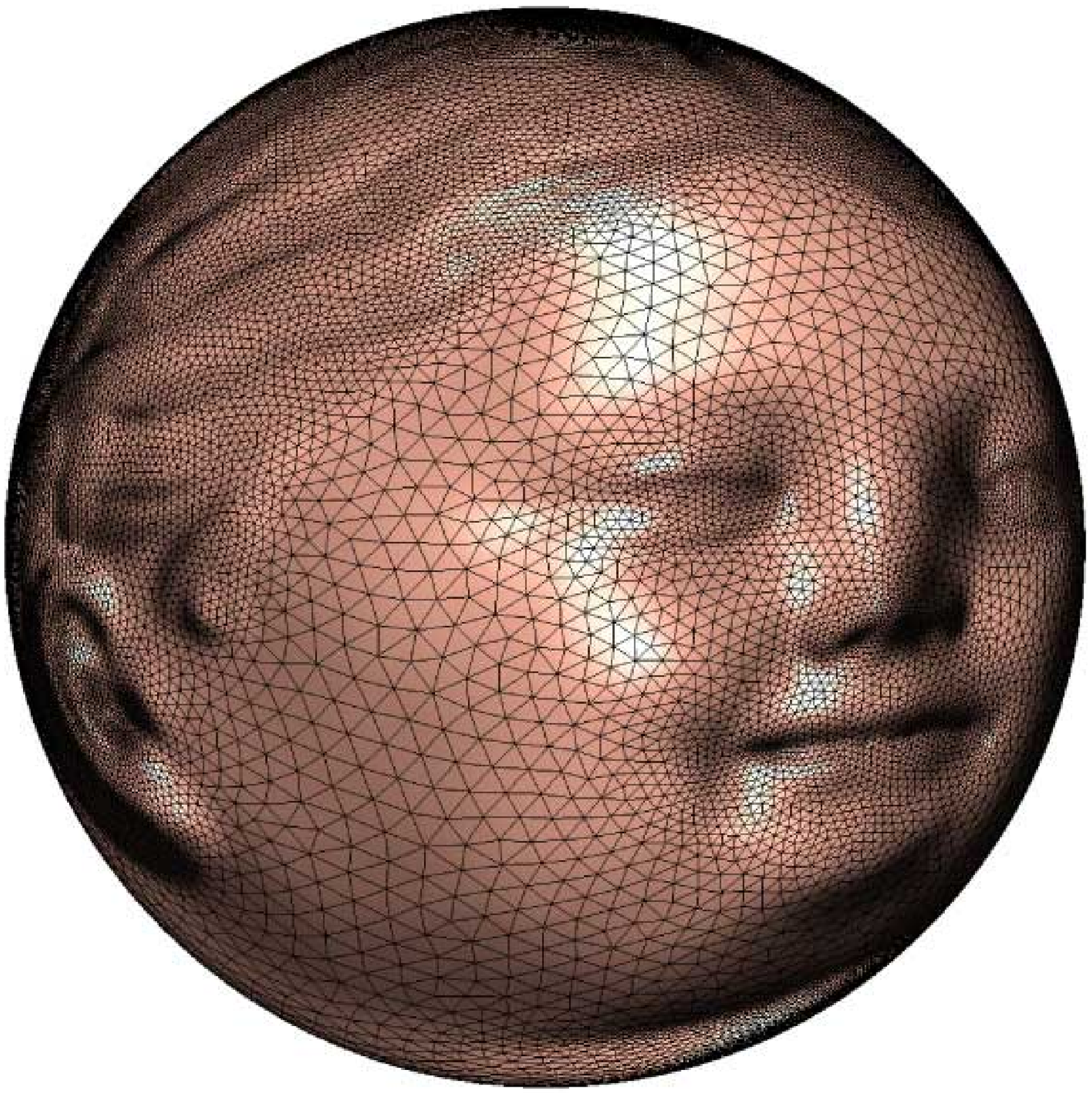}&
\includegraphics[width=1.0in]{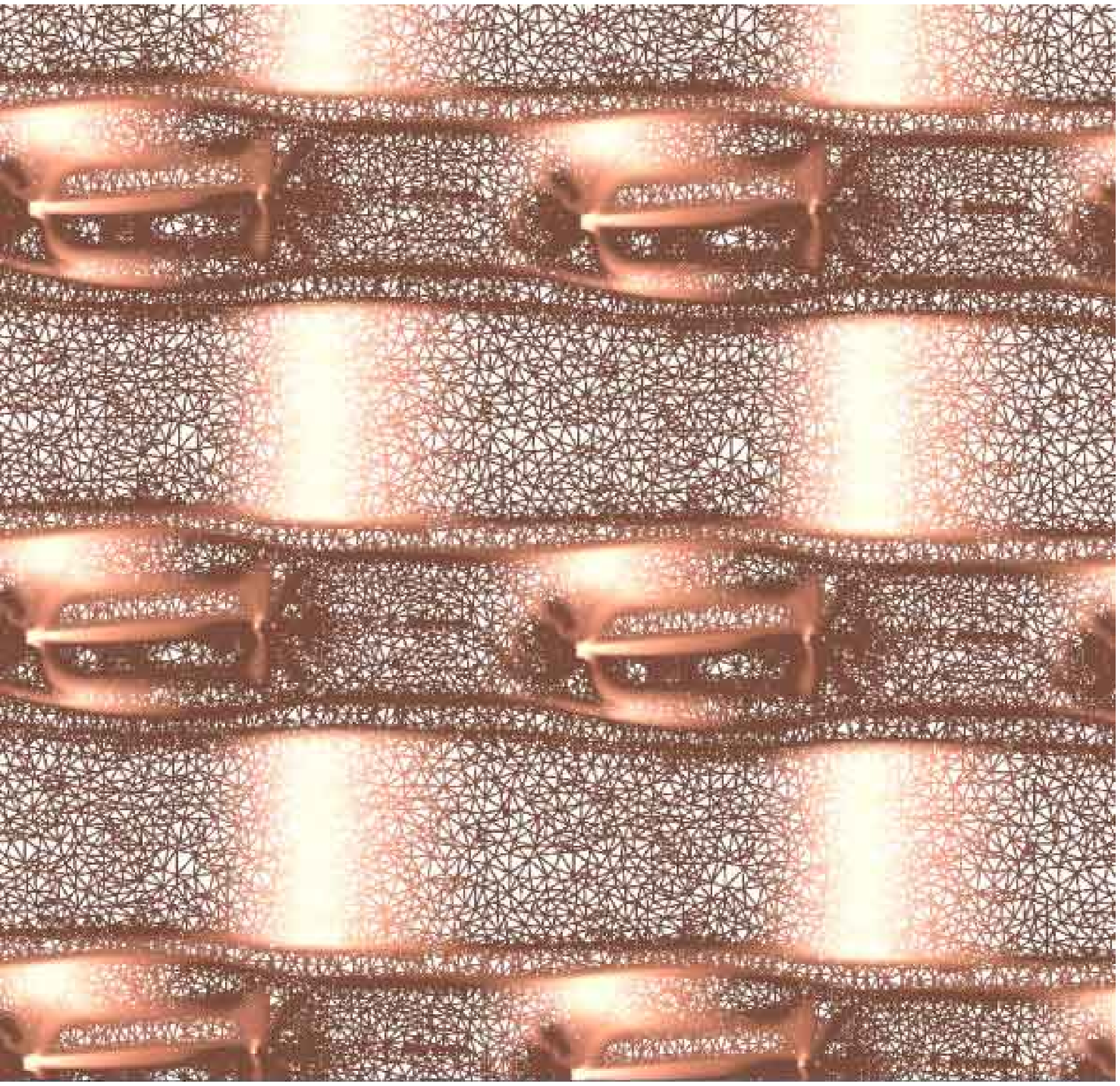}&
\includegraphics[width=1.0in]{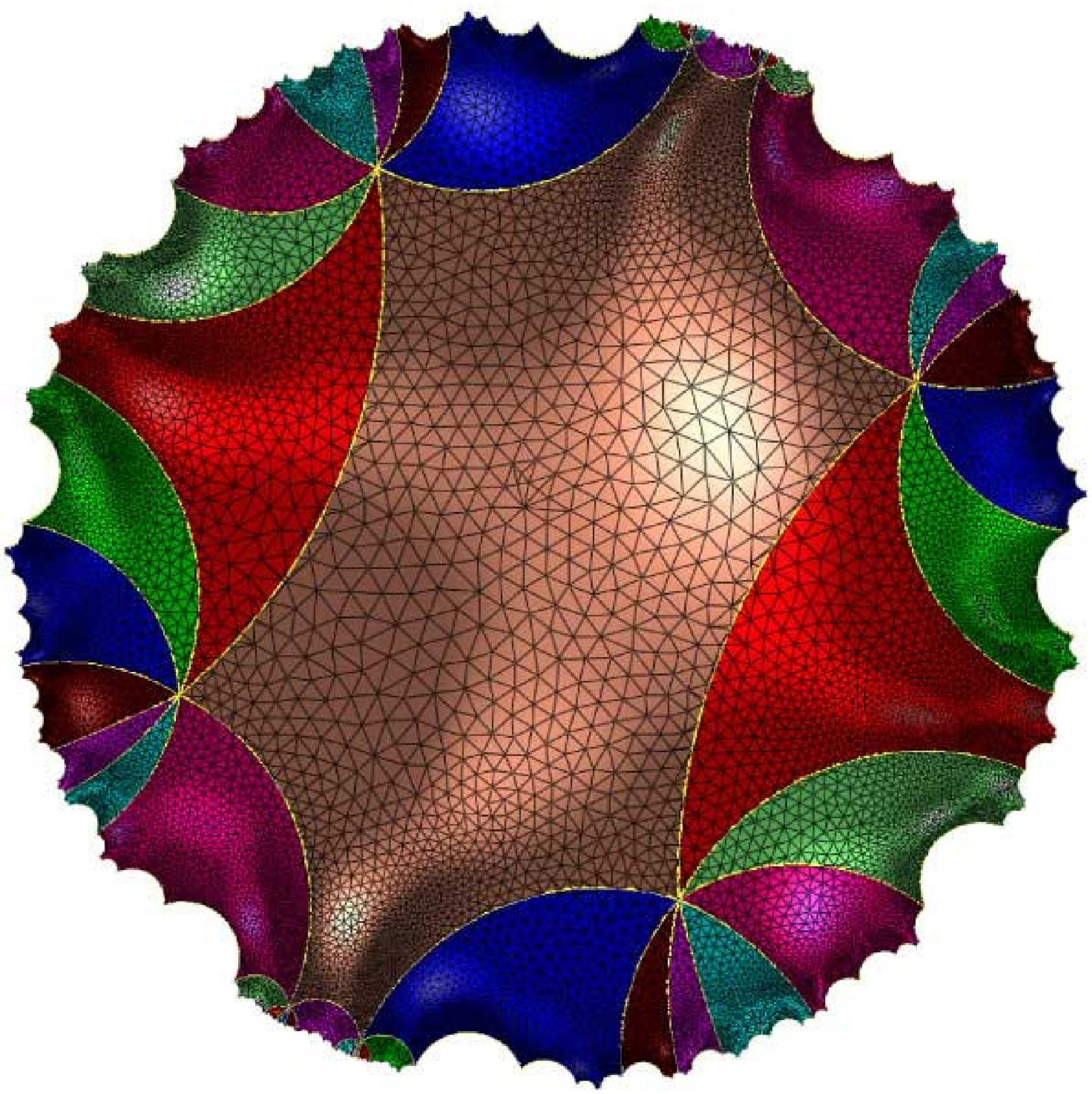}\\
\end{tabular}
\caption{Uniformization for closed surfaces by Ricci flow.\label{fig:uniformization_closed_surfaces}}
\end{center}
\end{figure}

\begin{figure}
\begin{center}
\begin{tabular}{ccc}
\includegraphics[width=1.0in]{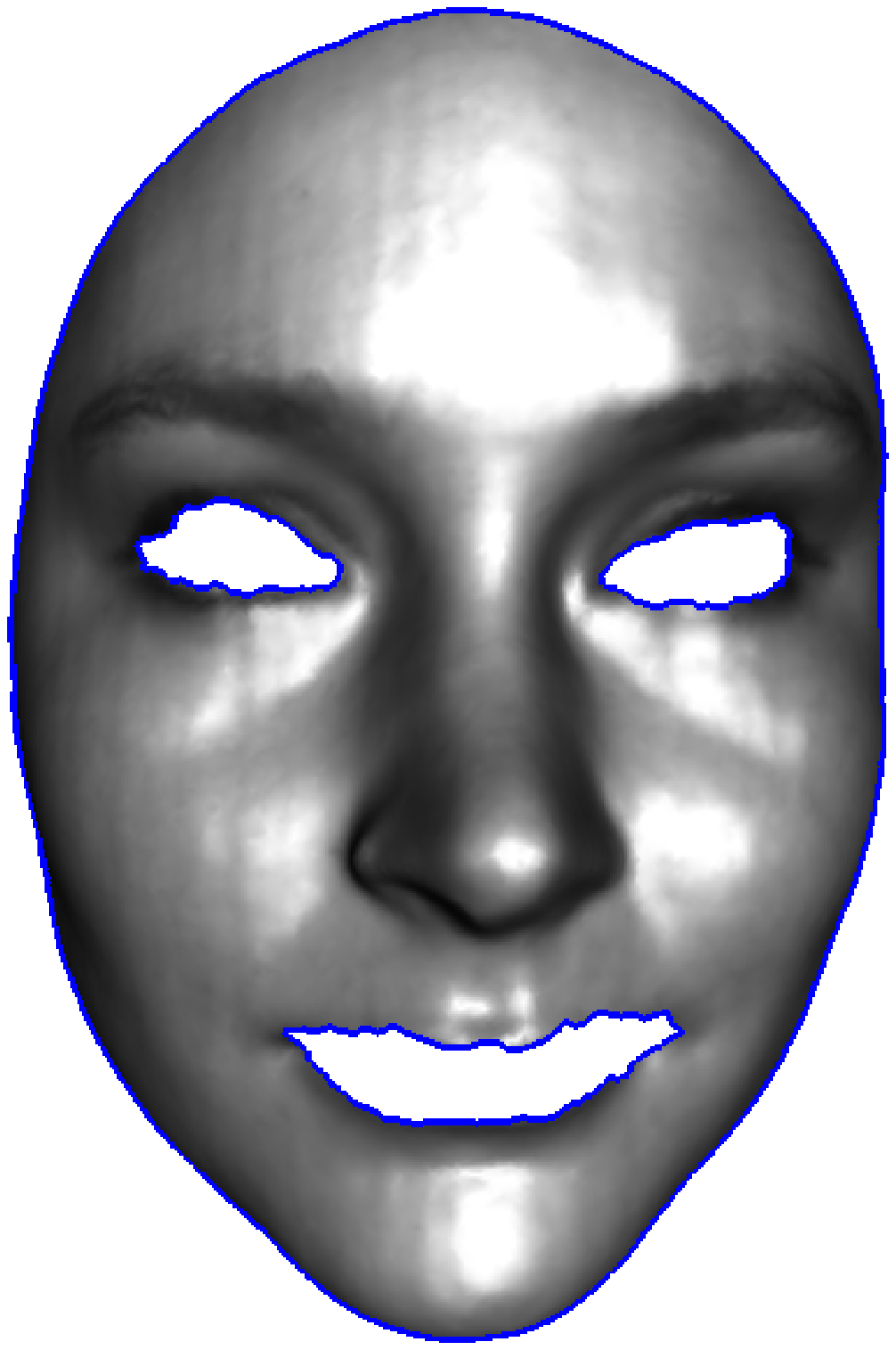}&
\includegraphics[width=1.0in]{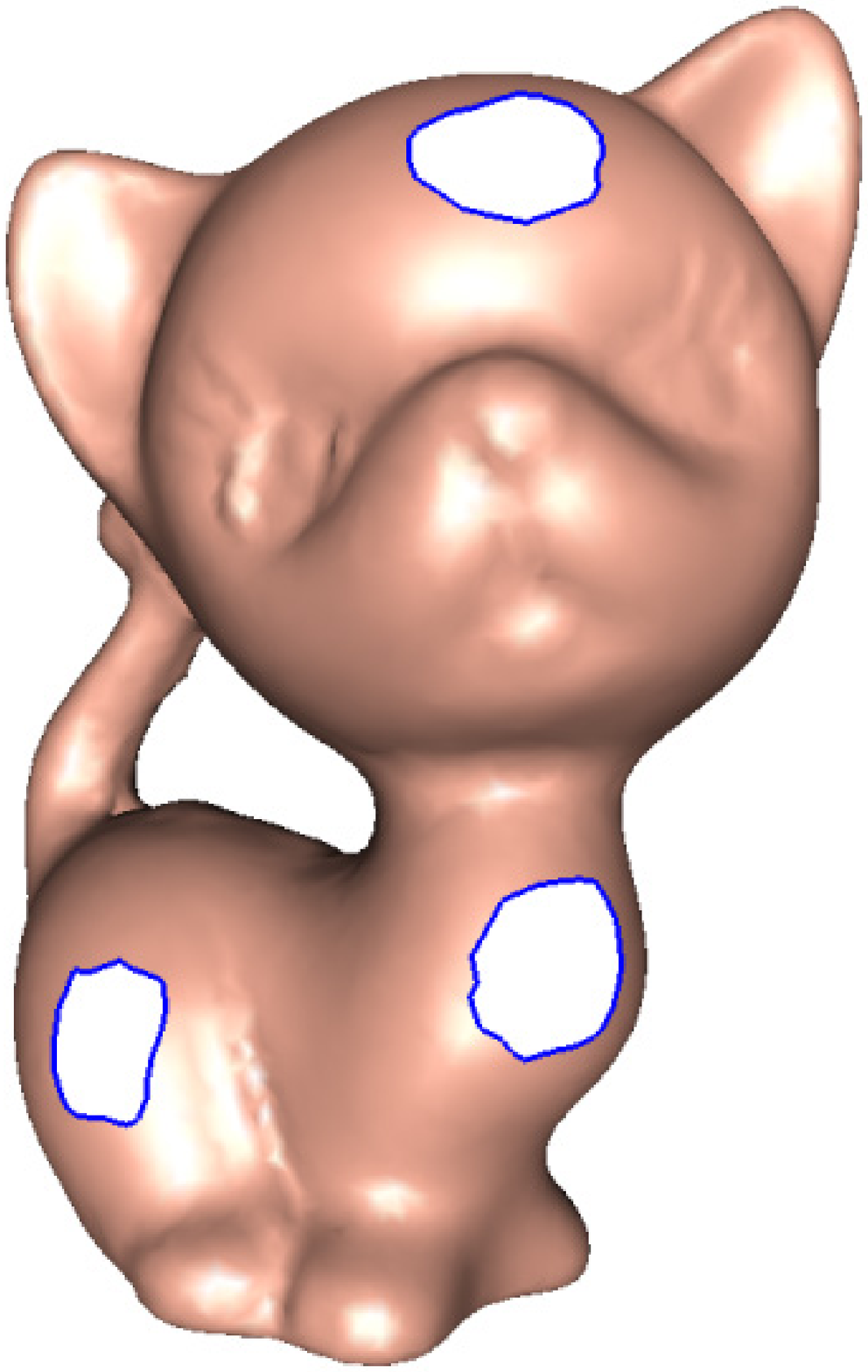}&
\includegraphics[width=1.0in]{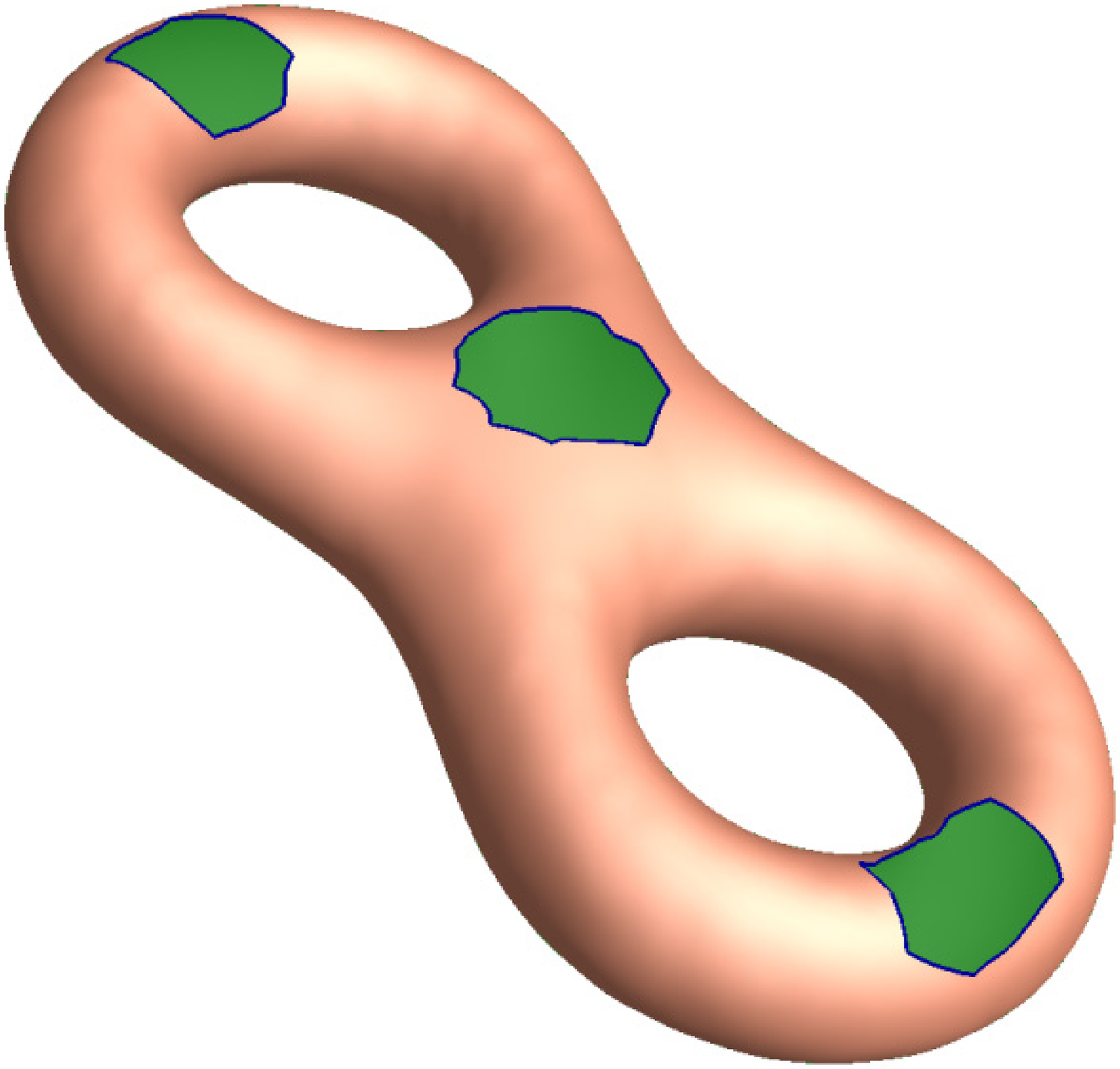}\\
\includegraphics[width=1.0in]{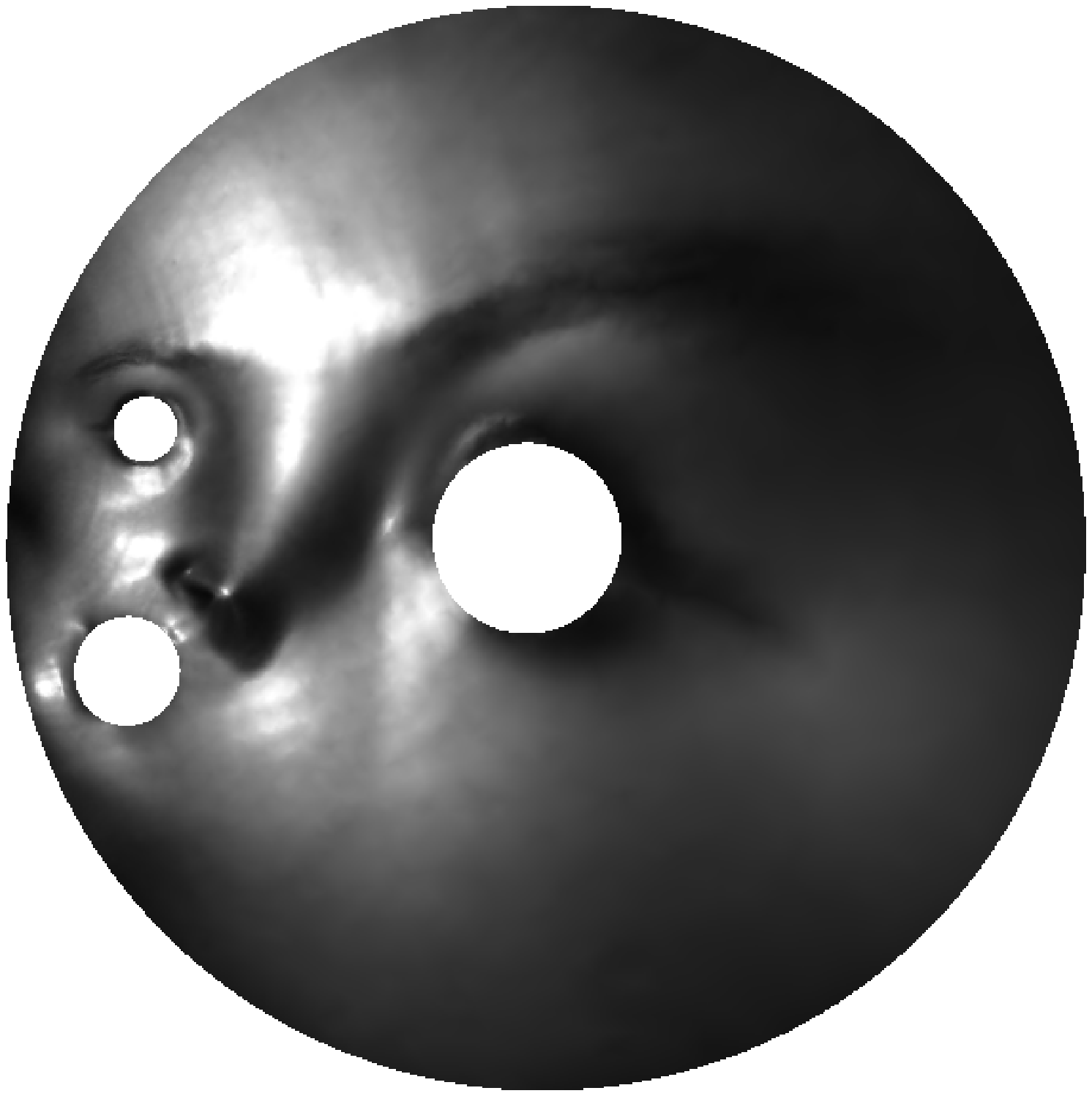}&
\includegraphics[width=1.0in]{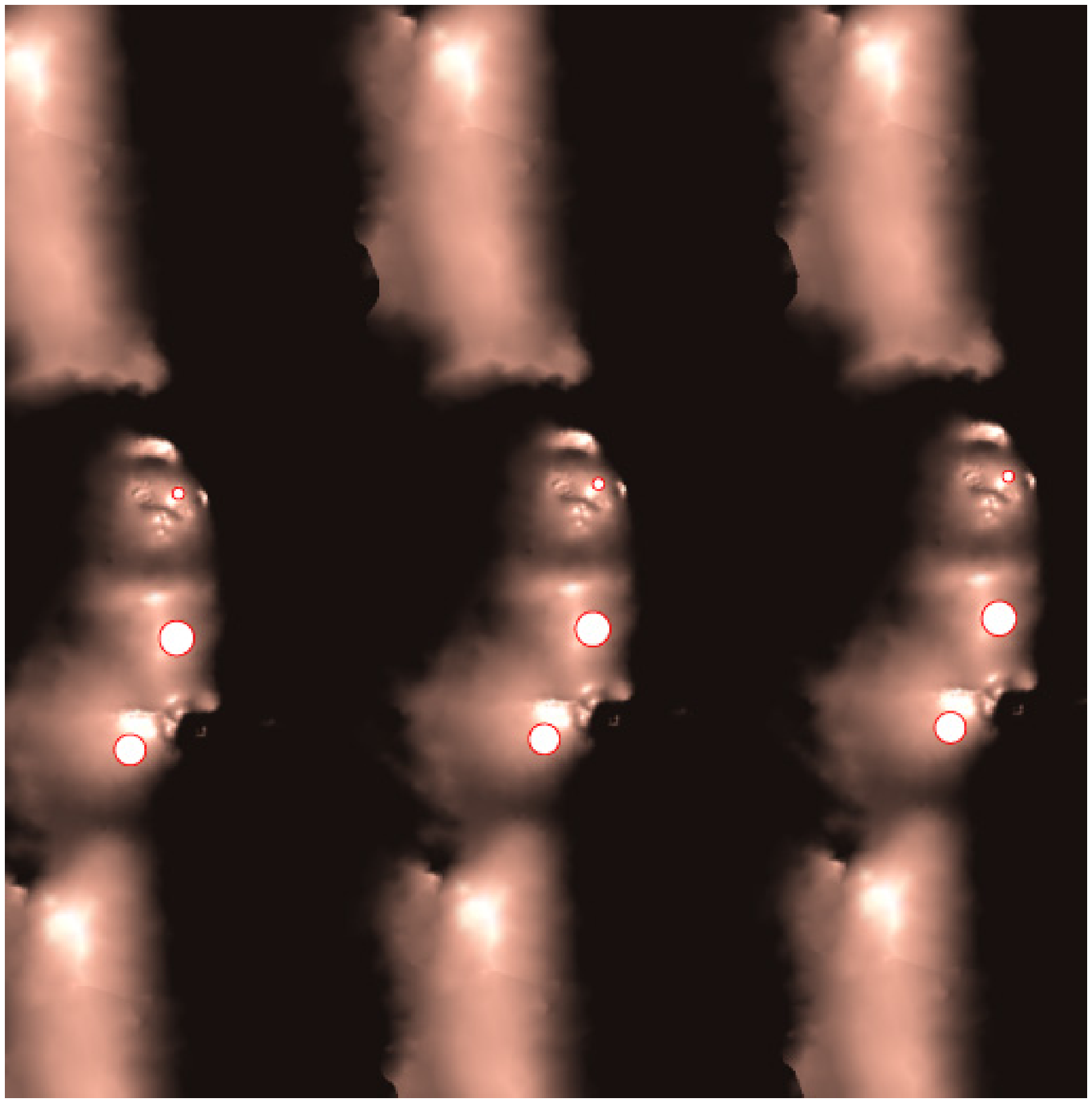}&
\includegraphics[width=1.0in]{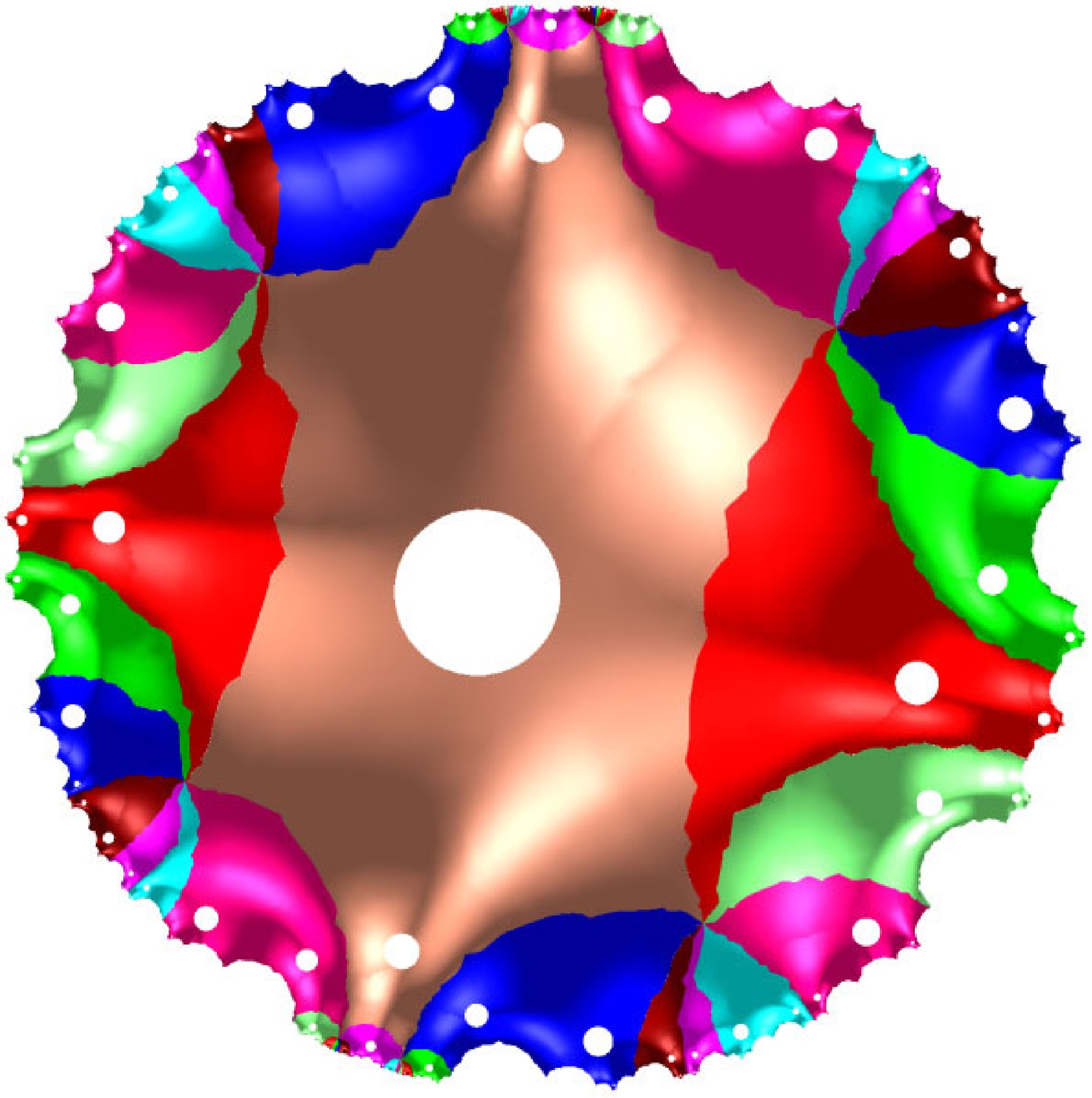}\\
\end{tabular}
\caption{Uniformization for surfaces with boundaries by Ricci flow.\label{fig:uniformization_open_surfaces}}
\end{center}
\end{figure}

\begin{figure}[t!]
\centering
\includegraphics[width=0.25\textwidth]{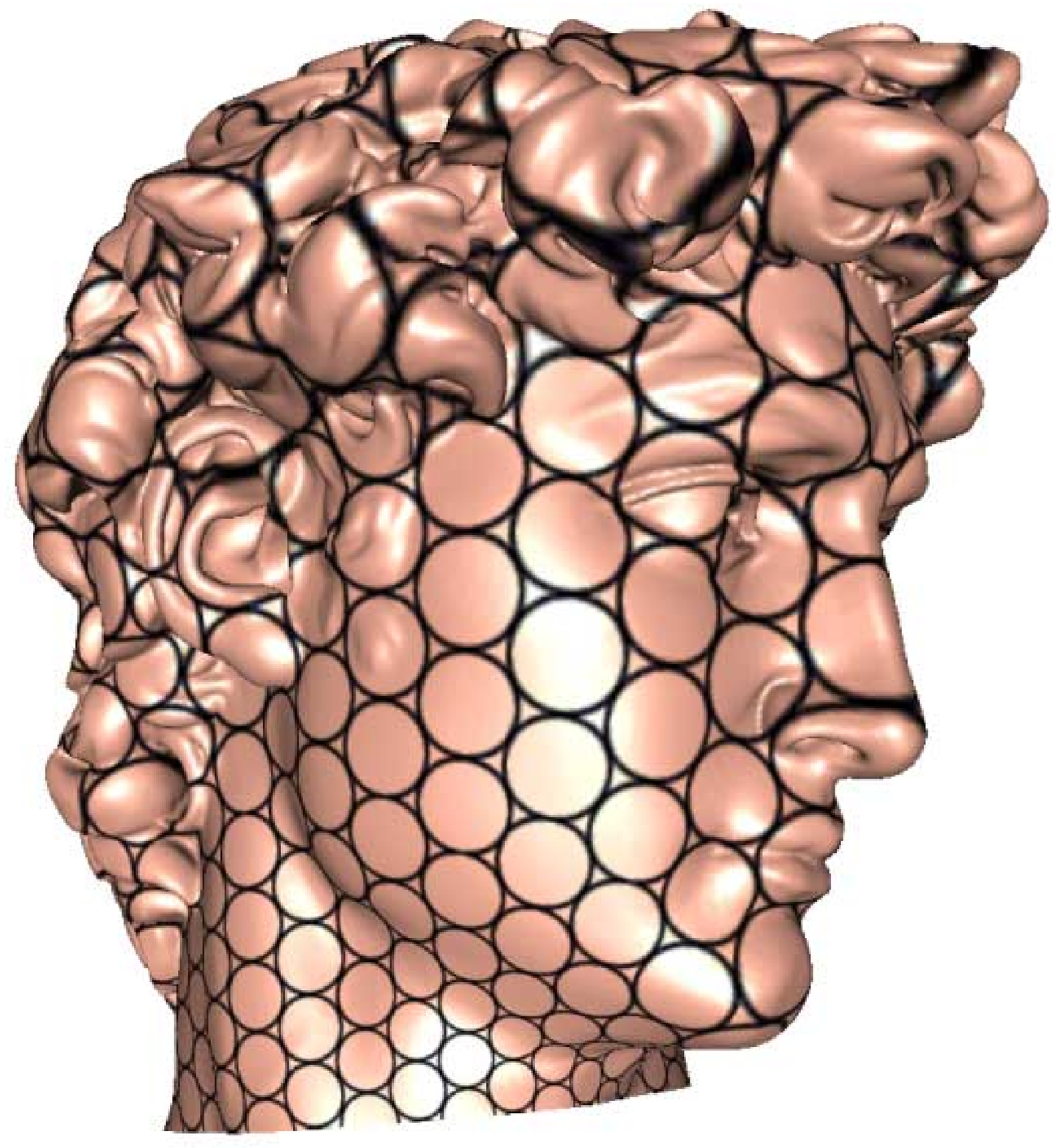}
\includegraphics[width=0.25\textwidth]{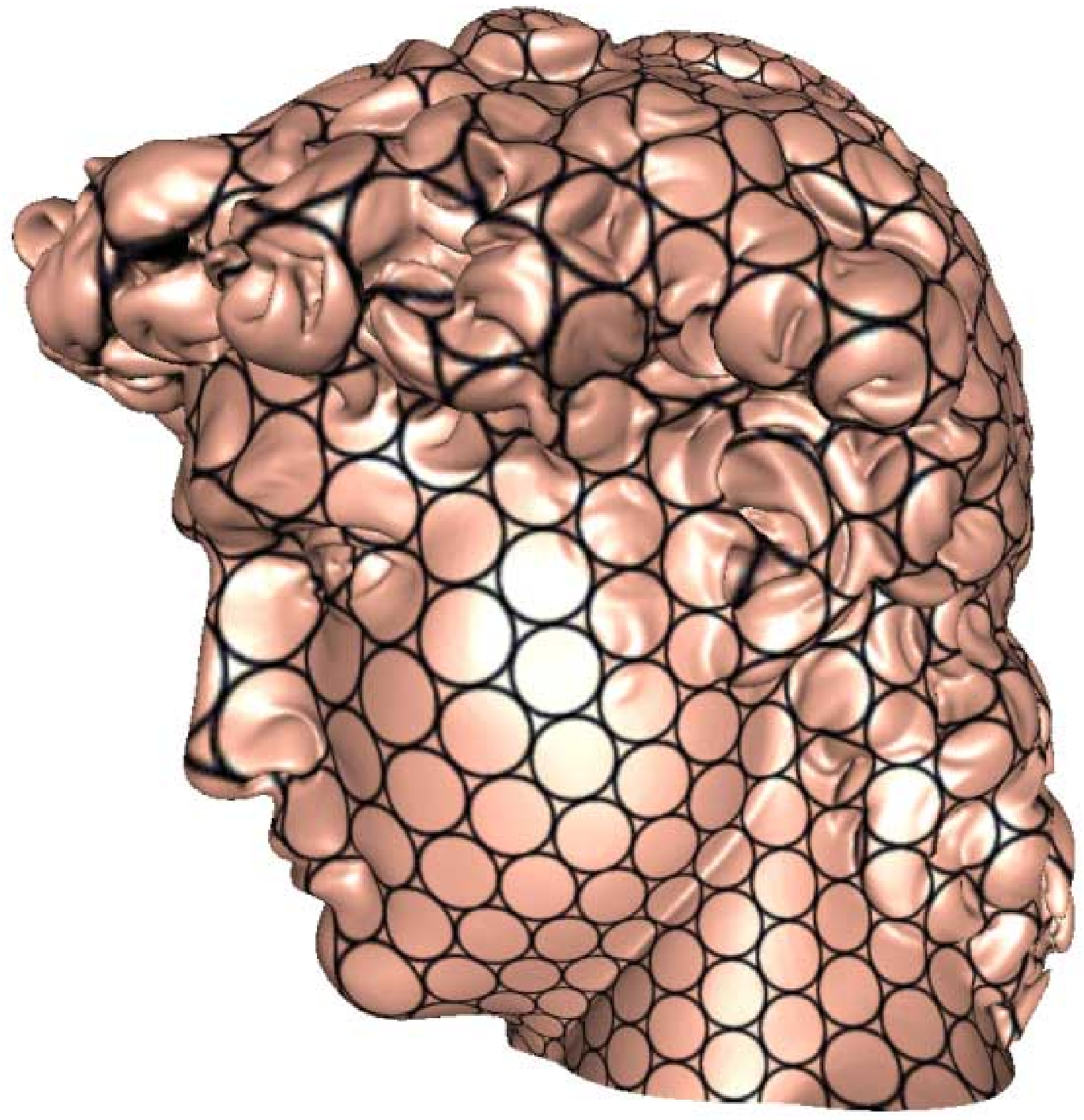}
\caption{Conformal mapping preserves infinitesimal circles.\label{fig:conformal_mapping}}
\vspace{-0.5cm}
\end{figure}

Conformal metric deformation transforms infinitesimal circles to
infinitesimal circles as shown in Fig.~\ref{fig:conformal_mapping}.
Intuitively, one approximates the surface by a triangulated
polyhedron (a triangle mesh), covers each vertex by a disk of finite
size (a cone), and deforms the disk radii preserving the
combinatorial structure of the triangulation and the intersection
angles among the circles. This deformation simulates the smooth
conformal mapping with very high fidelity. Rodin and Sullivan
\cite{Sullivan87} proved that if the triangulation of a simply
connected planar domain is subdivided infinite times, the induced
discrete conformal mappings converge to the smooth Riemann mapping.
The discrete version of surface Ricci flow was introduced by Chow
and Luo in \cite{chow_luo_03} in 2003. It is based on the circle
packing method.

Historically, many schemes of circle packing or circle pattern have
been invented. The discrete surface can be constructed by gluing
Spherical, Euclidean or Hyperbolic triangles isometrically along
their edges. Accordingly, we say the triangle mesh has spherical
$\mathbb{S}^2$, Euclidean $\mathbb{E}^2$ or hyperbolic
$\mathbb{H}^2$ background geometry. Under each background geometry,
there are $6$ schemes, tangential circle packing, Thurston's circle
packing, inversive distance circle packing, discrete Yamabe flow,
virtual radius circle packing and mixed type scheme. There are $18$
combinations in total. Among them, the hyperbolic and spherical
virtual radius circle packing and mixed type schemes are first
introduced in this work.

Most of the existing schemes were invented and developed individually in the past. This work seeks a coherent theoretic framework, which can unify all the existing schemes, and predicts undiscovered ones. This leads to deeper understandings of discrete surface Ricci flow and provides approaches for further generalization. In practice, the theoretic discovery of virtual radius circle packing gives novel computational algorithm; the mixed schemes improves the flexibility; the unified framework greatly simplifies the implementation; the geometric interpretations offer better intuitions.

\subsection{Contributions}

This work has the following contributions:

\begin{enumerate}
\item This work establishes a unified framework for discrete surface Ricci flow, which covers most existing schemes: tangential circle packing, Thurston's circle packing, inversive distance circle packing, discrete Yamabe flow, virtual radius circle packing and mixed type schemes, in Spherical, Euclidean and hyperbolic background geometry. In Euclidean case, our unified framework is equivalent to Glickenstein's geometric formulation \cite{Glickenstein_2011}.To the best of our knowledge, the unified frameworks for both hyperbolic and spherical schemes are reported for the first time.

\item This work introduces $4$ novel schemes for discrete surface Ricci flow: virtual radius circle packing and mixed type schemes under both hyperbolic and Euclidean background geometries, which are naturally deduced from our unification work. To the best of our knowledge, these are introduced to the literature for the first time.

\item This work gives an explicit geometric interpretation to the discrete Ricci energy for all the $18$ schemes. The geometric interpretations to $2$ Yamabe flow schemes (both Euclidean and Hyperbolic) were first made by Bobenko, Pinkall and Springborn \cite{Bobenko_Pinkall_Springborn_2010}.

\item This work also provides an explicit geometric interpretation
to the Hessian of discrete Ricci energy for all the $18$ schemes.
The interpretation in Euclidean case is due to
Glickenstein~\cite{Glickenstein_2011}. To the best of our knowledge,
the interpretation in Hyperbolic and Spherical cases are introduced
for the first time. Recently, Glickenstein and Thomas discovered the
similar result independently \cite{GlickensteinPreprint}.

\end{enumerate}

The paper is organized as follows: section \ref{sec:review} briefly
reviews the most related theoretic works; section
\ref{sec:unified_discrete_surface_ricci_flow}introduces the unified
framework for different schemes of discrete surface Ricci flow,
which covers $18$ schemes in total; section \ref{sec:Hessian}
explains the geometric interpretation of the Hessian matrix of
discrete Ricci energy for all schemes with different background
geometries; section \ref{sec:energy} gives a geometric
interpretation of Ricci energy; The work concludes in section \ref{sec:conclusion}, future
directions are discussed; Finally, in the appendix, we give the
implementation details and reorganize all the formulae.

\section{Previous Works}
\label{sec:review}

\paragraph{Thurston's Circle Packing}

In his work on constructing hyperbolic metrics on 3-manifolds,
Thurston \cite{Thurston97} studied a Euclidean (or a hyperbolic)
circle packing on a triangulated closed surface with prescribed
intersection angles. His work generalizes Koebe's and Andreev's
results of circle packing on a sphere
\cite{Andreev_1970_1,Andreev_1970_2,Koebe_1936}. Thurston
conjectured that the discrete conformal mapping based on circle
packing converges to the smooth Riemann mapping when the discrete
tessellation becomes finer and finer. Thurston's conjecture has been
proved by Rodin and Sullivan \cite{Sullivan87}. Chow and Luo
established the intrinsic connection between circle packing and
surface Ricci flow \cite{chow_luo_03}.

The rigidity for classical circle packing was proved by Thurston
\cite{Thurston97}, Marden-Rodin \cite{Marden_Rodin_1990}, Colin de
Verdi\'{e}re \cite{Colin_de_Verdiere_1991}, Chow-Luo
\cite{chow_luo_03}, Stephenson \cite{stephenson05}, and He
\cite{He96}.

\paragraph{Inversive Distance Circle Packing}

Bowers-Stephenson \cite{Bowers_Stephenson_2004}
introduced inversive distance circle packing which generalizes
Andreev-Thurston's intersection angle circle packing. See Stephenson
\cite{stephenson05} for more information. Guo gave a proof for local rigidity \cite{Guo_2011} of inversive distance circle packing. Luo gave a proof for global rigidity in \cite{Luo11GeomTop}.

\paragraph{Yamabe Flow}

Luo introduced and studied the combinatorial Yamabe
problem for piecewise flat metrics on triangulated surfaces
\cite{Luo_2004}. Springborn, Schr\"oder and Pinkall
\cite{Springborn_Schroder_Pinkall_2008} considered this
combinatorial conformal change of piecewise flat metrics and found
an explicit formula of the energy function. Glickenstein
\cite{Glickenstein_2005a,Glickenstein_2005b} studied the
combinatorial Yamabe flow on 3-dimensional piecewise flat manifolds. Bobenko-Pinkall-Springborn introduced a geometric interpretation to Euclidean and hyperbolic Yamabe flow using the volume of generalized hyperbolic tetrahedron in \cite{Bobenko_Pinkall_Springborn_2010}.
Combinatorial Yamabe flow on hyperbolic surfaces with boundary has
been studied by Guo in \cite{Guo_Yamabe_2011}. The existence of the solution to Yamabe flow with topological surgeries
has been proved recently in \cite{Gu_Luo_Sun_Wu_13} and \cite{Gu_Guo_Luo_Sun_Wu_14}.

\paragraph{Virtual Radius Circle Packing}

The Euclidean virtual radius circle packing first appeared in \cite{ricci2013}. The hyperbolic and spherical virtual radius circle packing are introduced in this work.

\paragraph{Mixed type Circle Packing}

The Euclidean mixed type circle packing appeared in \cite{ricci2013} and Glickenstein's talk \cite{Glickenstein_Talk}. This work introduces hyperbolic and spherical mixted type schemes.

\paragraph{Unified Framework}

Recently Glickenstein \cite{Glickenstein_2011} set the theory of
combinatorial Yamabe flow of piecewise flat metric in a broader
context including the theory of circle packing on surfaces.
This work focuses on the hyperbolic and spherical unified frameworks.

\paragraph{Variational Principle}

The variational approach to circle packing was first introduced by
Colin de Verdi\'{e}re \cite{Colin_de_Verdiere_1991}. Since then,
many works on variational principles on circle packing or circle
pattern have appeared. For example, see Br\"agger
\cite{Bragger_1992}, Rivin \cite{Rivin_1994}, Leibon
\cite{Leibon_2002}, Chow-Luo \cite{chow_luo_03}, Bobenko-Springborn
\cite{Bobenko_Springborn_2004}, Guo-Luo \cite{Guo_Luo_2009}, and
Springborn \cite{Springborn_2008}. Variational principles for
polyhedral surfaces including the topic of circle packing were
studied systematically in Luo \cite{Luo_06}. Many energy functions
are derived from the cosine law and its derivative. Tangent circle
packing is generalized to tangent circle packing with a family of
discrete curvature. For exposition of this work, see also Luo-Gu-Dai
\cite{VPDS07}.

\paragraph{Discrete Uniformization}

Recently, Gu et al established discrete uniformization theorem based
on Euclidean \cite{Gu_Luo_Sun_Wu_13} and hyperbolic
\cite{Gu_Guo_Luo_Sun_Wu_14} Yamabe flow.
 In a series of papers on developing
discrete uniformization theorem
\cite{Hersonsky_2012_1},\cite{Hersonsky_2012_2},\cite{Hersonsky_2011}
and \cite{Hersonsky_2008}, Sa'ar Hersonsky proved several important
theorems based on discrete harmonic maps and cellular
decompositions. His approach is complementary to the work mentioned
above.

\section{Unified Discrete Surface Ricci Flow}
\label{sec:unified_discrete_surface_ricci_flow}

This section systematically introduces the unified framework for
discrete surface Ricci flow. The whole theory is explained using the
variational principle on discrete surfaces based on derivative
cosine law \cite{VPDS07}. The elementary concepts and
some of schemes can be found in \cite{Luo_06} and the chapter 4 in \cite{ricci2013}.

\subsection{Elementary Concepts\index{discrete surface}}

In practice, smooth surfaces are usually approximated by
\emph{discrete surfaces}. Discrete surfaces are represented as two
dimensional simplicial complexes which are manifolds, as shown in
Fig. \ref{fig:discrete_surface}.

\begin{figure}
\begin{center}
\begin{tabular}{cc}
\includegraphics[width=0.3\textwidth]{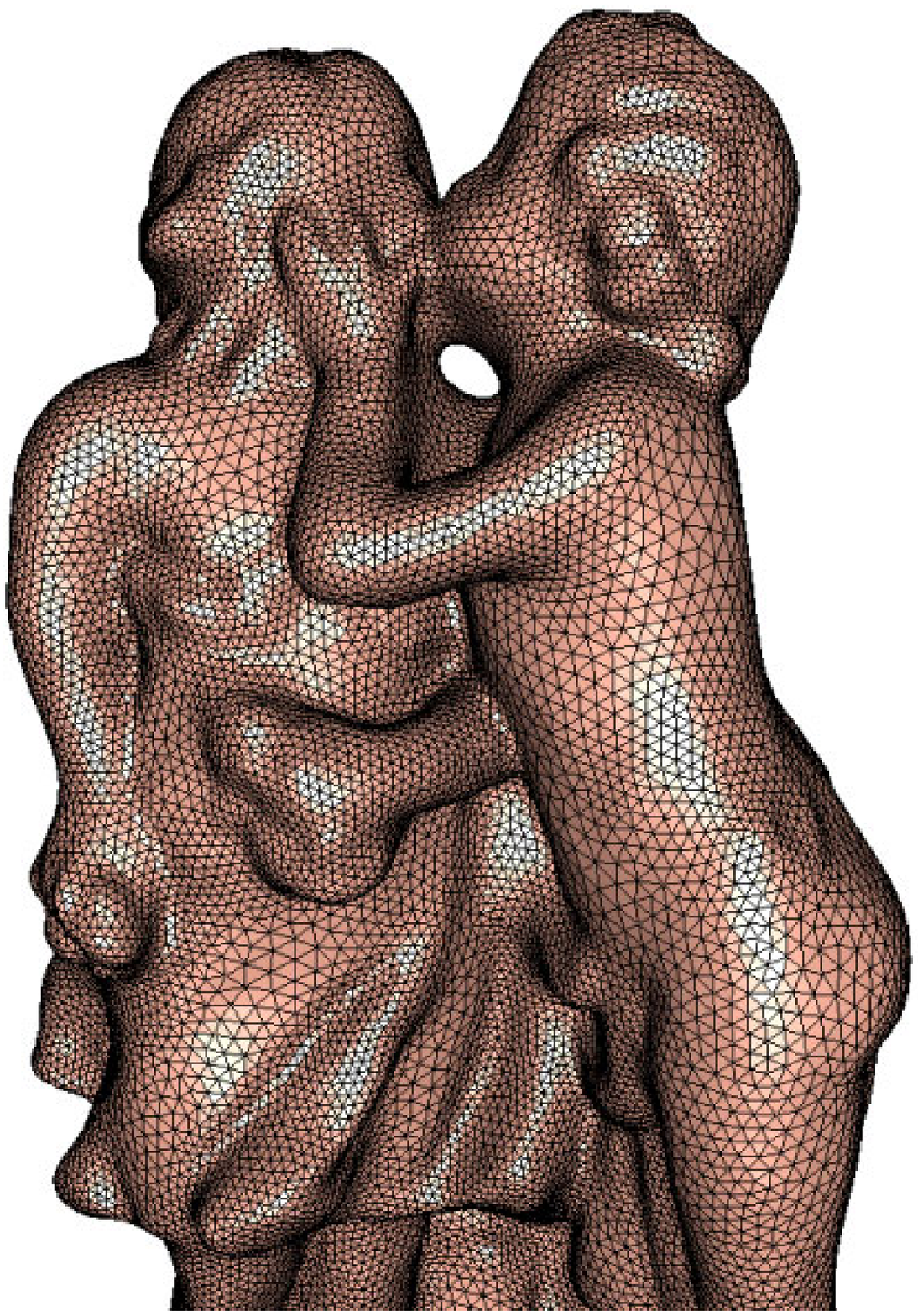}&
\includegraphics[width=0.3\textwidth]{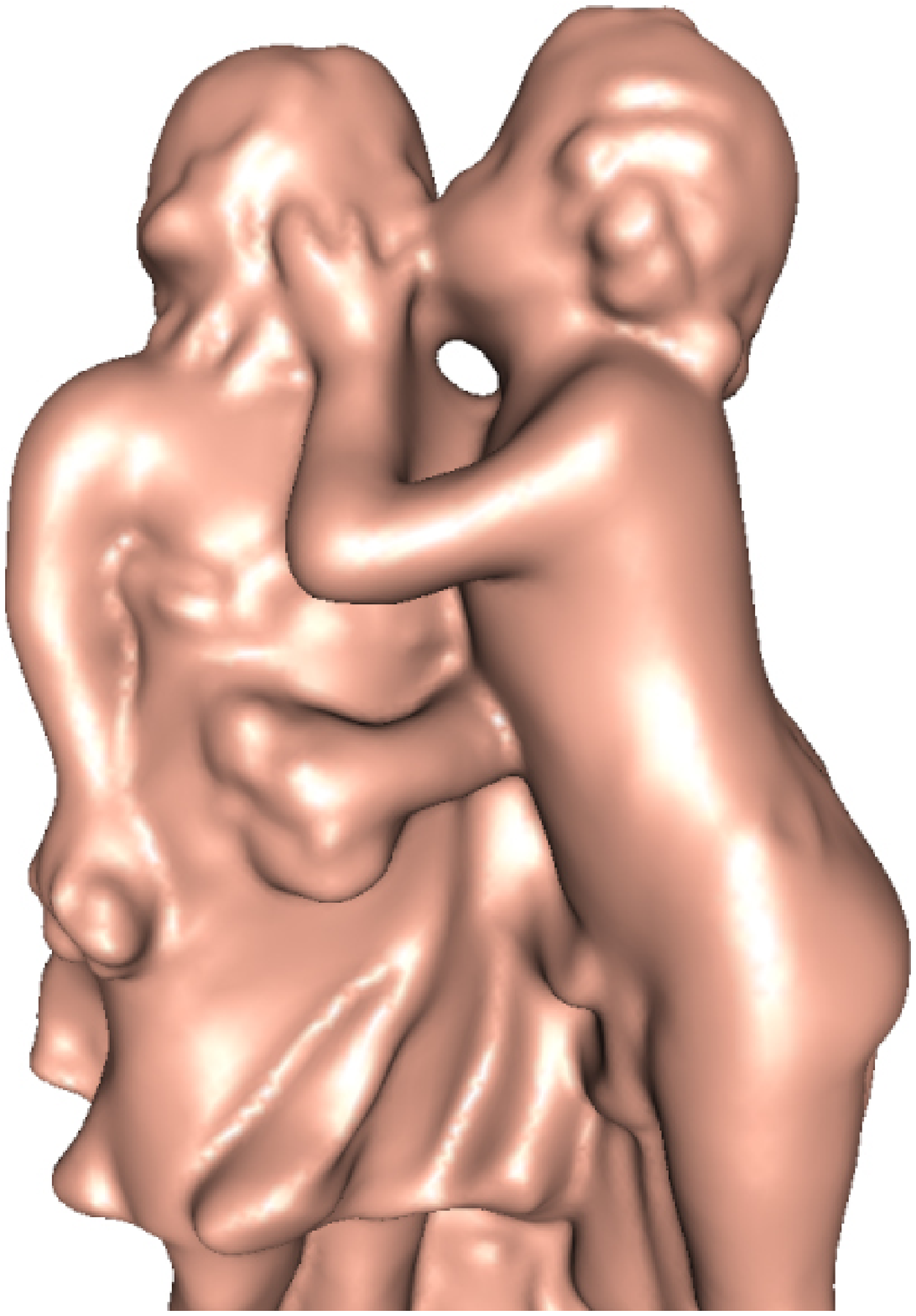}\\
\end{tabular}
\caption{Smooth surfaces are approximated by discrete Surfaces\label{fig:discrete_surface}}
\end{center}
\end{figure}

\begin{definition}[Triangular Mesh\index{triangular mesh}] Suppose $\Sigma$ is a two dimensional simplicial complex, furthermore it is also a manifold, namely, for each point $p$ of $\Sigma$, there exists a neighborhood of $p$, $U(p)$, which is homeomorphic to the whole plane or the upper half plane. Then $\Sigma$ is called a triangular mesh.

If $U(p)$ is homeomorphic to the whole plane, then $p$ is called an interior point; if $U(p)$ is homeomorphic to the upper half plane, then $p$ is called a boundary point.
\end{definition}

The fundamental concepts from smooth
differential geometry, such as Riemannian metric, curvature and
conformal structure, are generalized to the simplicial complex, respectively.

In the following discussion, we use $\Sigma=(V,E,F)$ to denote the mesh
with vertex set $V$, edge set $E$ and face set $F$. A discrete surface is with Euclidean (hyperbolic or spherical) background geometry if it is constructed by isometrically gluing triangles in $\mathbb{E}^2$ ($\mathbb{H}^2$ or $\mathbb{S}^2$).

\begin{definition}[Discrete Riemannian Metric\index{discrete Riemannian metric}]
A discrete metric on a triangular mesh is a function defined on the
edges, $l: E \to \mathbb{R}^+$, which satisfies the triangle inequality: on
each face $[v_i,v_j,v_k]$, $l_i,l_j,l_k$ are the lengths of edges against $v_i,v_j,v_k$ respectively,
\[
    l_i+l_j > l_k, ~ l_j+l_k > l_i,~ l_k+l_i > l_j.
\]
A triangular mesh with a discrete Riemannian metric is called a discrete metric surface.
\end{definition}

\begin{figure}[h!]
\begin{center}
\begin{tabular}{ccc}
\includegraphics[width=0.2\textwidth]{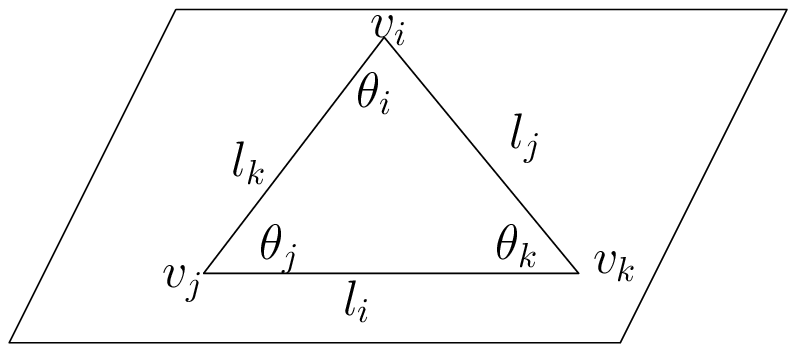}&
\includegraphics[width=0.2\textwidth]{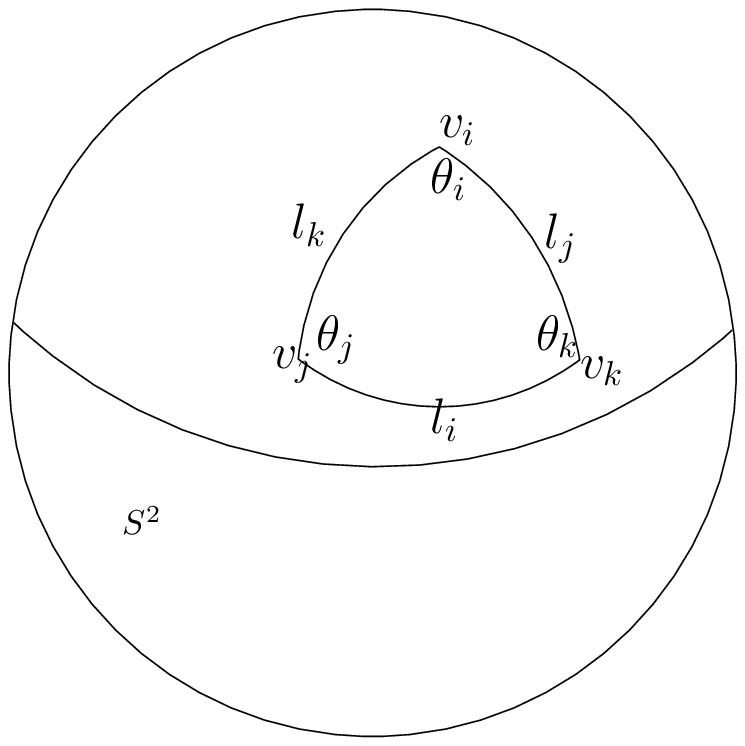}&
\includegraphics[width=0.2\textwidth]{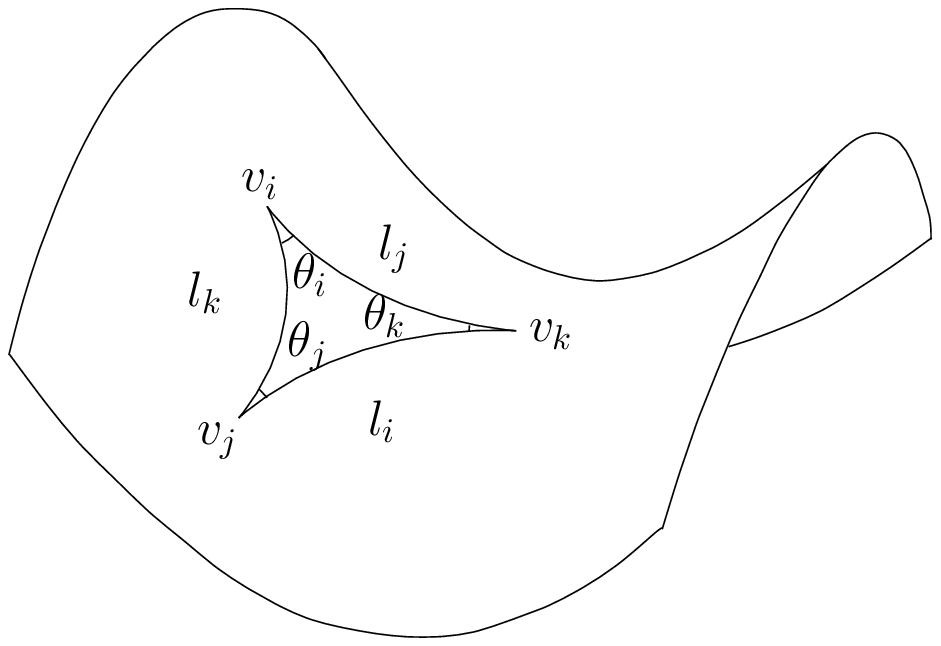}\\
$\mathbb{E}^2$ & $\mathbb{S}^2$ & $\mathbb{H}^2$\\
\end{tabular}
\caption{Different background geometry, Euclidean, spherical and hyperbolic.\label{fig:background_geometry}}
\end{center}
\end{figure}

\begin{definition}[Background Geometry]
Suppose $\Sigma$ is a discrete metric surface, if each face of $\Sigma$ is a spherical, ( Euclidean or hyperbolic ) triangle,
then we say $\Sigma$ is with spherical, (Euclidean or hyperbolic) background geometry.
We use $\mathbb{S}^2$, $\mathbb{E}^2$ and $\mathbb{H}^2$ to represent spherical Euclidean or hyperbolic background metric.
\end{definition}
Triangles with different background geometries satisify different cosine laws:
\[
    \begin{array}{rllr}
    1 &=& \frac{\cos\theta_i + \cos\theta_j \cos \theta_k}{\sin\theta_j \sin \theta_k} & \mathbb{E}^2 \\
    \cos l_i &=& \frac{\cos\theta_i + \cos\theta_j \cos \theta_k}{\sin\theta_j \sin \theta_k} & \mathbb{S}^2\\
    \cosh l_i &=& \frac{\cosh\theta_i + \cosh\theta_j \cosh \theta_k}{\sinh\theta_j \sinh \theta_k} & \mathbb{H}^2
    \end{array}
\]

The discrete Gaussian curvature is defined as angle deficit, as shown in Fig. \ref{fig:discrete_curvature}.

\begin{definition}[Discrete Gauss Curvature\index{discrete curvature}]
\label{def:discrete_curvature} The discrete Gauss curvature function on a
mesh is defined on vertices, $K: V \to \mathbb{R}$,
\[
K(v)= \left\{
\begin{array}{rl}
2\pi - \sum_{jk} \theta_i^{jk}, & v \not\in \partial M\\
 \pi - \sum_{jk} \theta_i^{jk}, & v \in \partial M\\
\end{array}
\right.,
\]
where $\theta_i^{jk}$'s are corner angle at $v_i$ in the face $[v_i,v_j,v_k]$, and $\partial M$ represents the boundary of the mesh.
\end{definition}

\begin{figure}[h!]
\begin{center}
\begin{tabular}{c}
\includegraphics[height=0.22\textwidth]{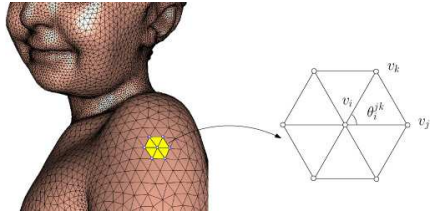}\\
\end{tabular}
\caption{Discrete curvatures of an interior vertex\label{fig:discrete_curvature}}
\end{center}
\end{figure}

The Gauss-Bonnet theorem still holds in the discrete case.

\begin{theorem}[Discrete Gauss-Bonnet Theorem]
Suppose $\Sigma$ is a triangular mesh with Euclidean background metric. The total curvature is a topological invariant,
\begin{equation}
    \sum_{v \not\in \partial \Sigma} K(v) + \sum_{v\in \partial \Sigma} K(v) + \epsilon A(\Sigma) =
    2\pi \chi(\Sigma),
    \label{eqn:Gauss_Bonnet}
\end{equation}
where $\chi$ is the characteristic Euler number, and $K$ is the Gauss curvature, $A(\Sigma)$ is the total area, $\epsilon=\{+1,0,-1\}$ if $\Sigma$ is with spherical, Euclidean or hyperbolic background geometry.
\label{thm:discrete_Gauss_Bonnet}
\end{theorem}

\subsection{Unified Circle Packing Metrics}

\begin{definition}[Circle Packing Metric\index{circle packing metric}]
Suppose $\Sigma=(V,E,F)$ is a triangle mesh with spherical,
Euclidean or hyperbolic background geometry. Each vertex $v_i$ is
associated with a circle with radius $\gamma_i$. The circle radius
function is denoted as $\gamma: V\to \mathbb{R}_{>0}$; a function
defined on the vertices $\epsilon: V\to \{+1,0,-1\}$ is called the
\emph{scheme coefficient}; a function defined on edges $\eta: E\to
\mathbb{R}$ is called the \emph{discrete conformal structure
coefficient}. A circle packing metric is a 4-tuple $(\Sigma, \gamma,
\eta, \epsilon)$, the edge length is determined by the 4-tuple and
the background geometry.
\end{definition}

In the smooth case, changing a Riemannian metric by a scalar
function, $\mathbf{g}\to e^{2u}\mathbf{g}$, is called a conformal
metric deformation. The discrete analogy to this is as follows.

\begin{definition}[Discrete Conformal Equivalence\index{discrete conformal equivalence}]
Two circle packing metrics $(\Sigma_k, \gamma_k,
\eta_k,\epsilon_k)$, $k=1,2$, are conformally equivalent if
$\Sigma_1=\Sigma_2$, $\eta_1=\eta_2$, $\epsilon_1=\epsilon_2$.
($\gamma_1$ may not equals to $\gamma_2$.)
\end{definition}

The discrete analogy to the concept of conformal factor in the smooth case is

\begin{definition}[Discrete Conformal Factor\index{discrete conformal factor}]
Discrete conformal factor for a circle packing metric $(\Sigma, \gamma,\eta,\epsilon)$ is a function defined
on each vertex $\mathbf{u} : V\to \mathbb{R}$,
\begin{equation}
    u_i =
    \left\{
    \begin{array}{lr}
    \log\gamma_i & \mathbb{E}^2\\
    \log\tanh \frac{\gamma_i}{2} & \mathbb{H}^2\\
    \log\tan \frac{\gamma_i}{2} & \mathbb{S}^2\\
    \end{array}
    \right.
    \label{eqn:discrete_conformal_factor}
\end{equation}
\end{definition}

\begin{figure}[t!]
\begin{center}
\begin{tabular}{ccc}
\includegraphics[width=0.2\textwidth]{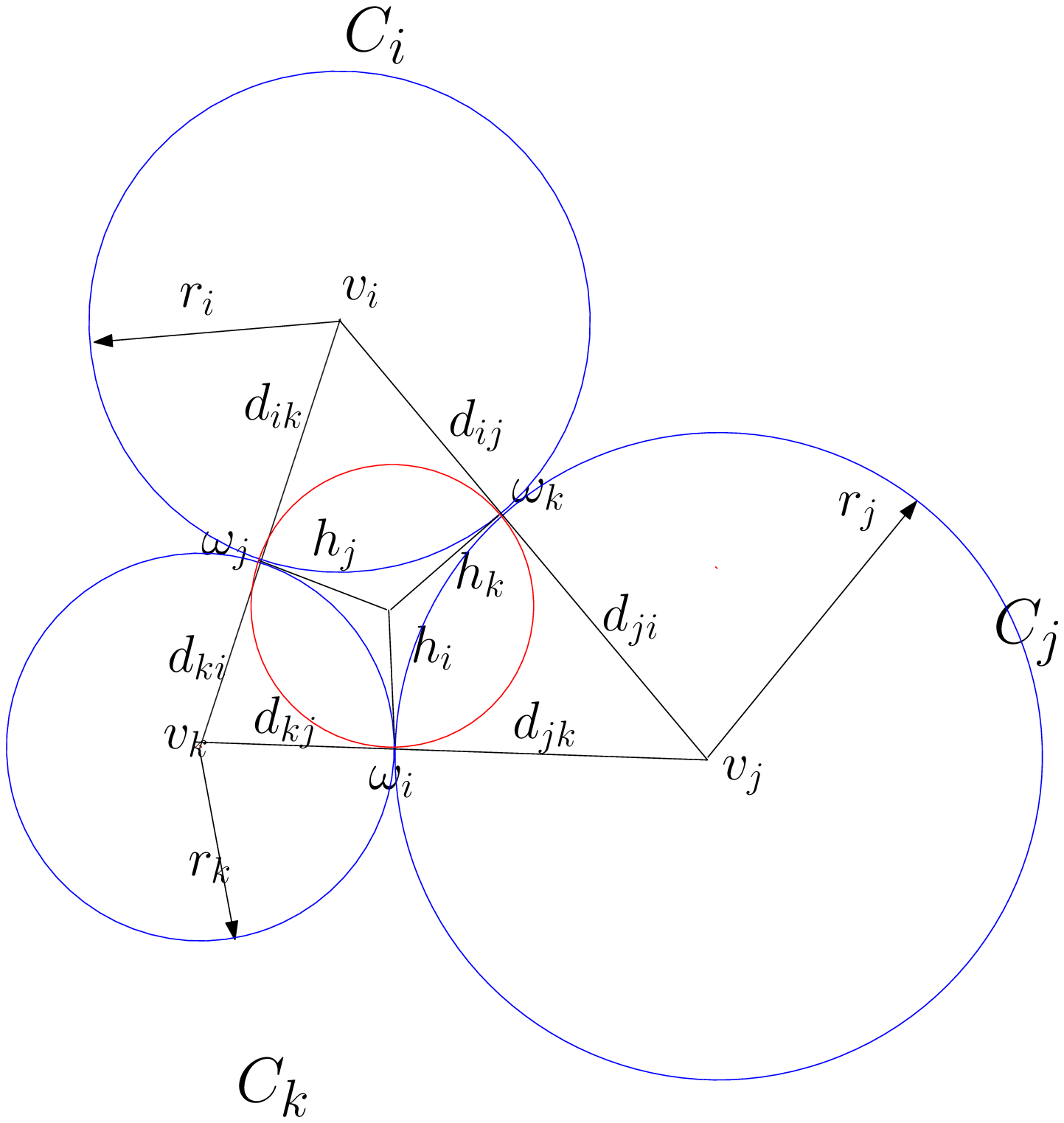}&
\includegraphics[width=0.2\textwidth]{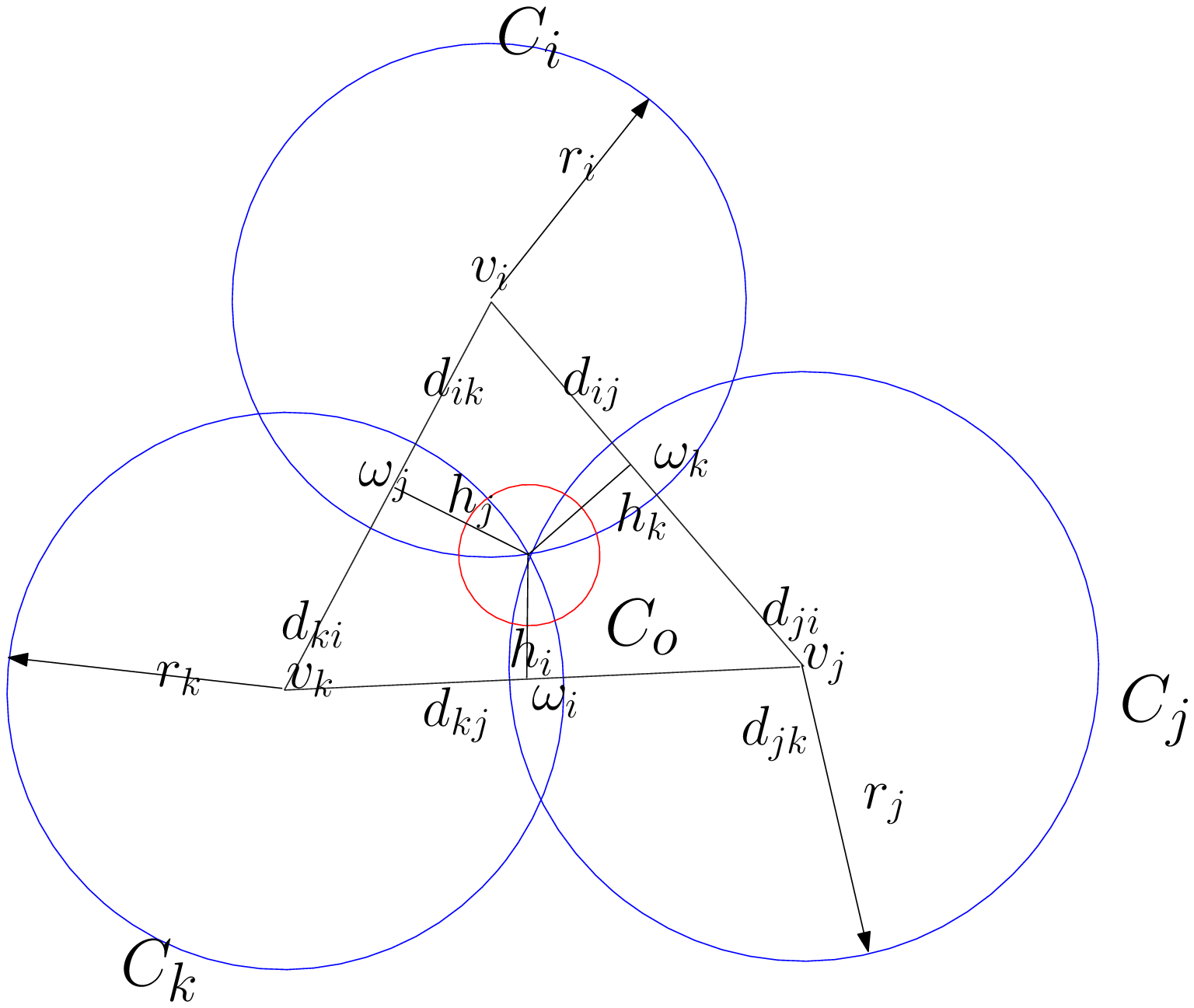}&
\includegraphics[width=0.22\textwidth]{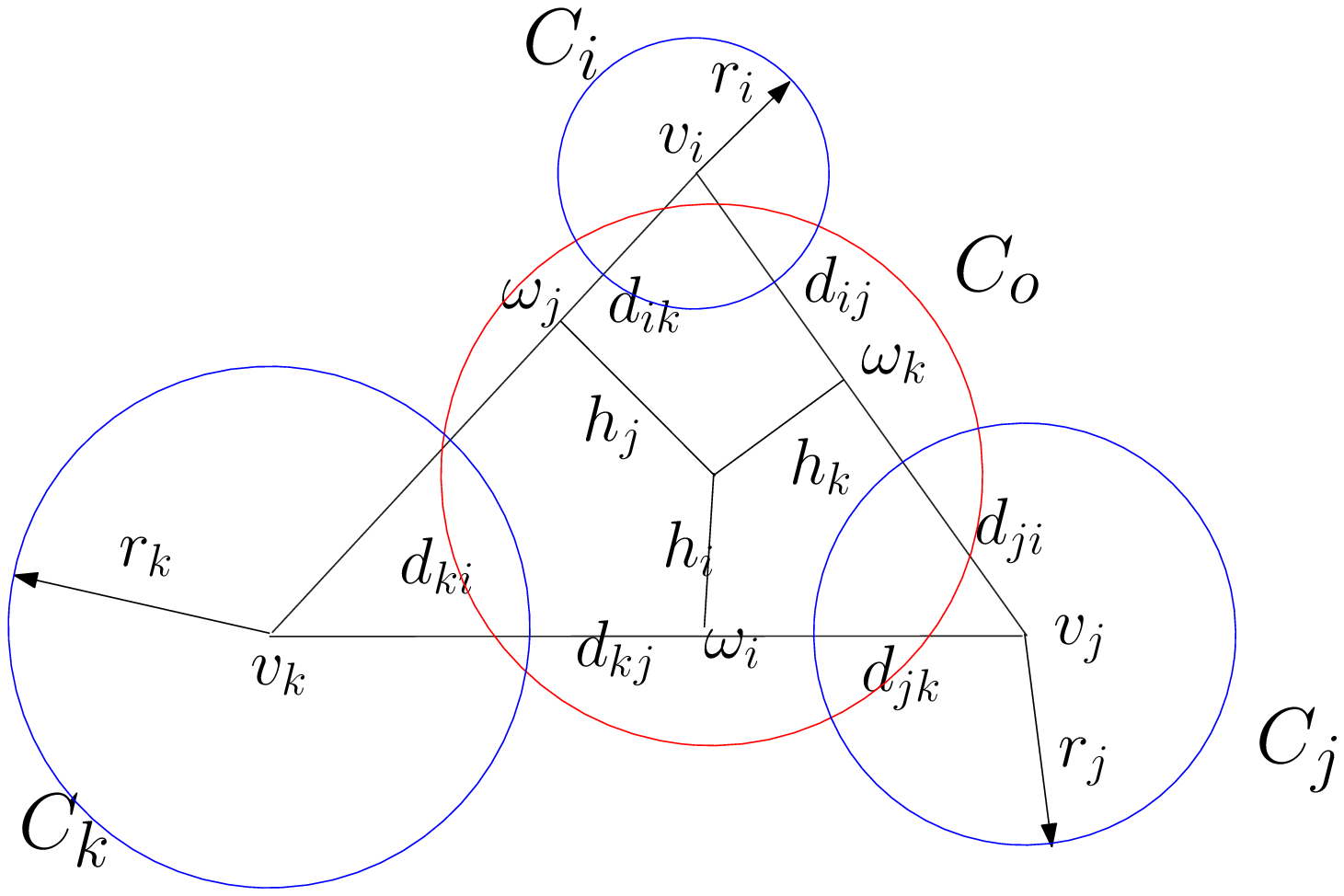}\\
(a) tangential CP & (b) Thurston's CP & (c) Inversive distance CP\\
\includegraphics[width=0.2\textwidth]{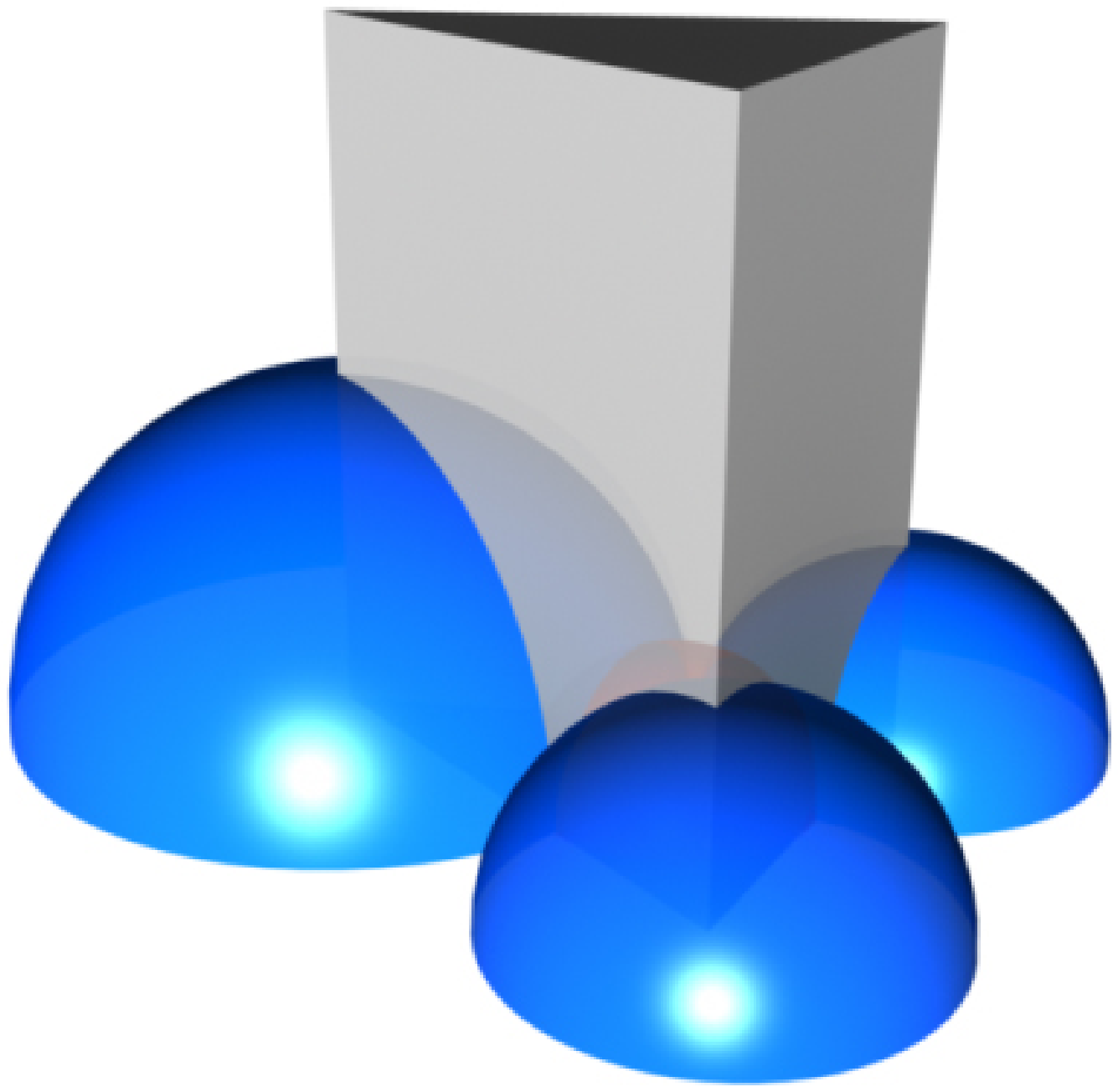}&
\includegraphics[width=0.2\textwidth]{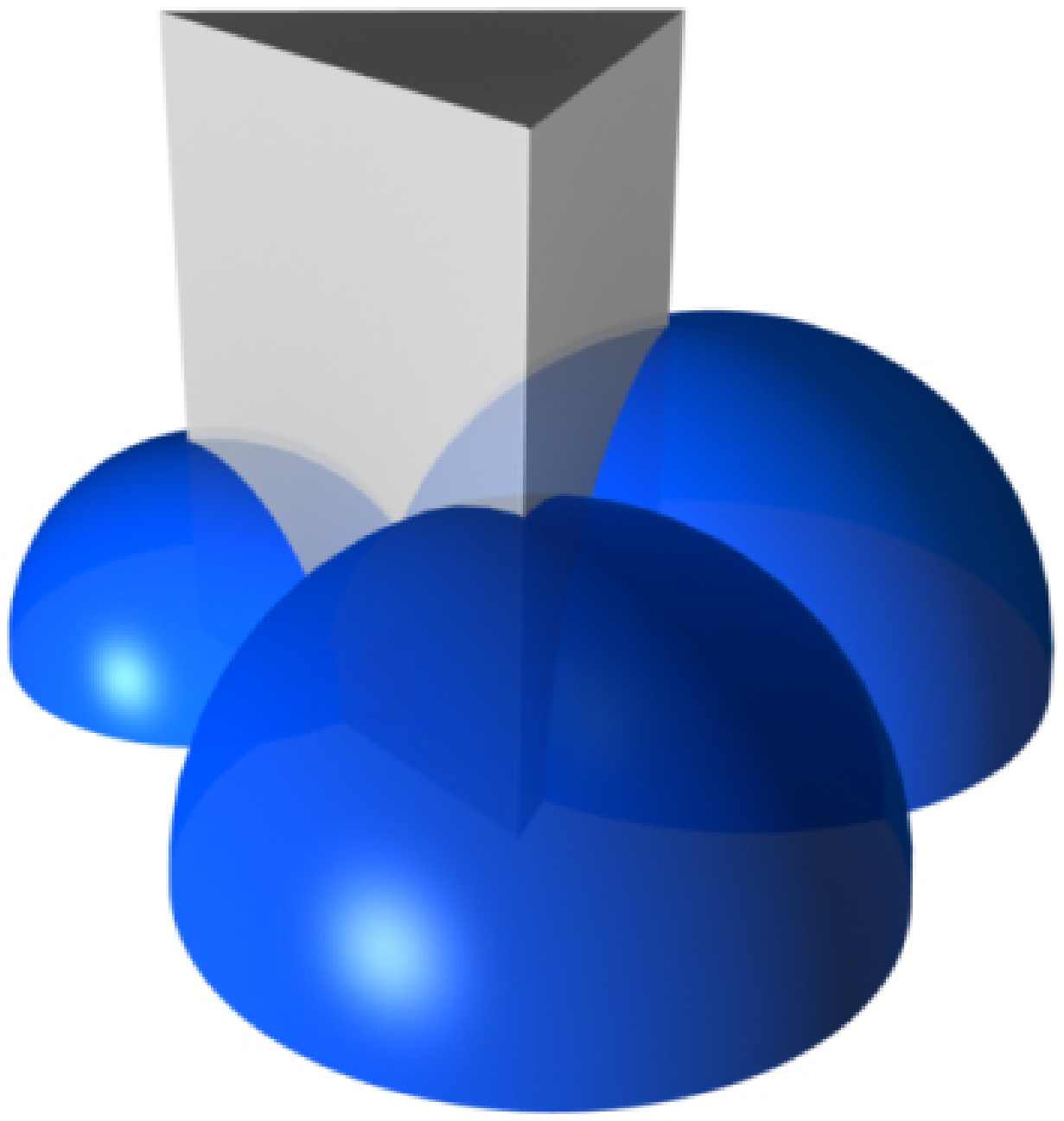}&
\includegraphics[width=0.2\textwidth]{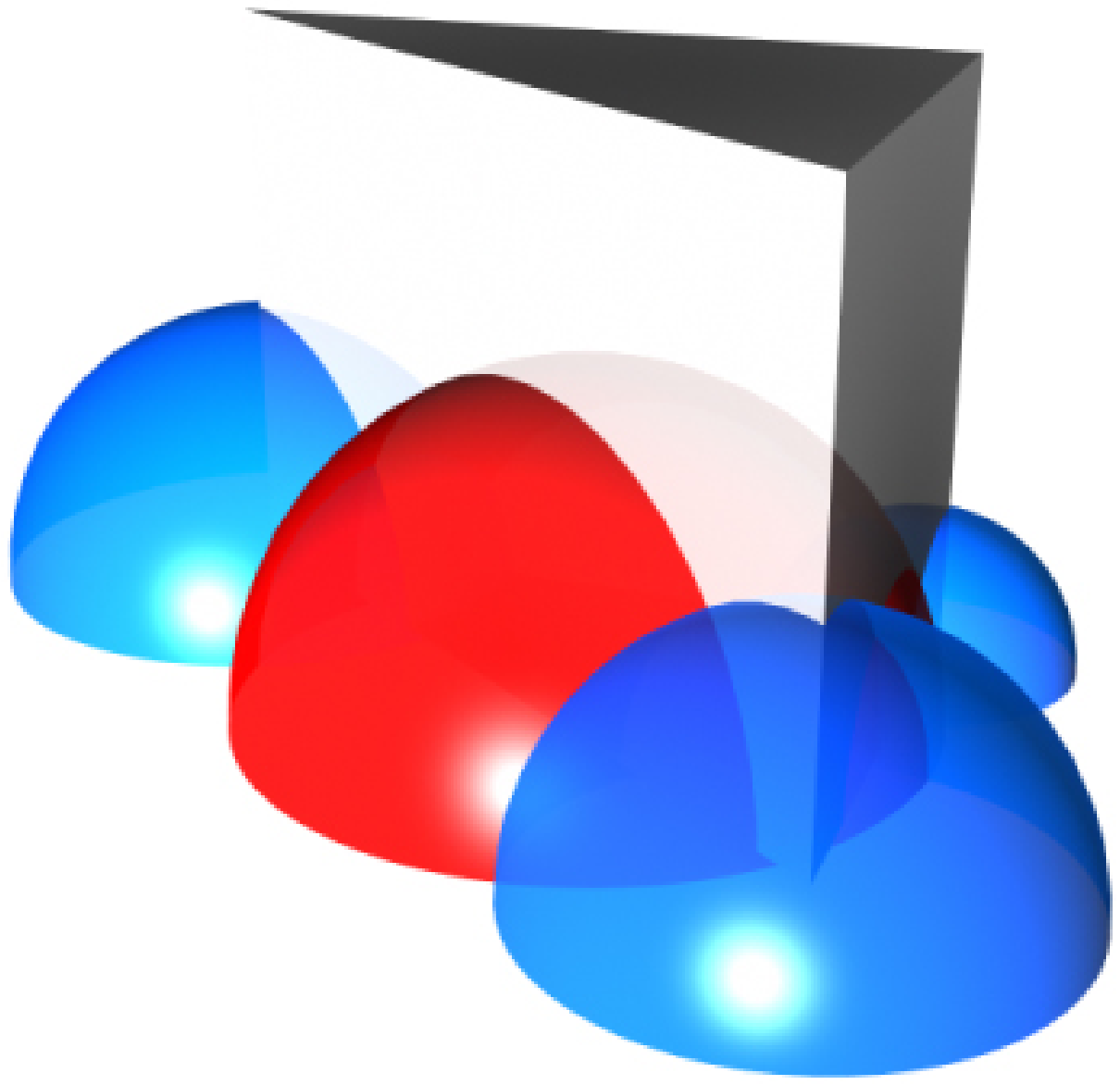}\\
$\eta=1,\epsilon=1$ & $ 0\le \eta \le 1,\epsilon=1$& $\eta\ge 1,\epsilon=1$ \\
\end{tabular}
\caption{Tangential circle packing, Thurston's circle packing and
inversive distance circle packing schemes, and the geometric
interpretations to their Ricci energies. \label{fig:schemes_1}}
\end{center}
\vspace{-0.5cm}
\end{figure}

\begin{figure}[t!]
\begin{center}
\begin{tabular}{ccc}
\includegraphics[width=0.2\textwidth]{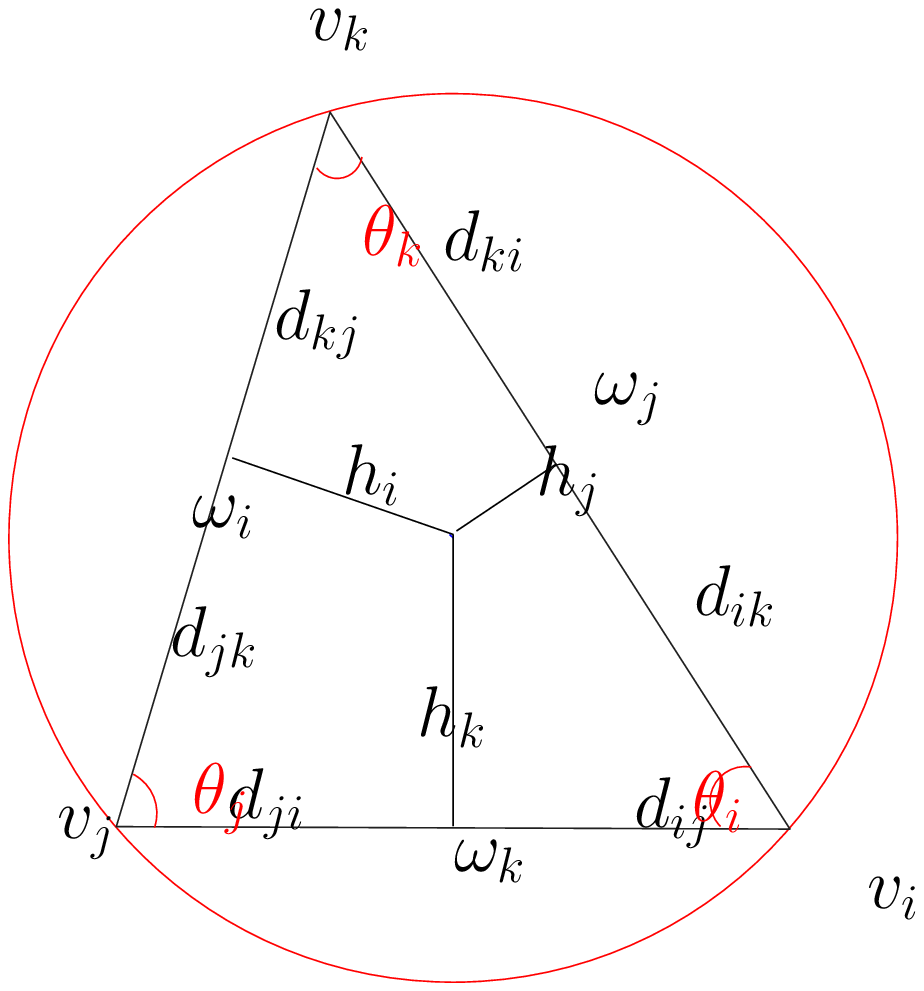}&
\includegraphics[width=0.23\textwidth]{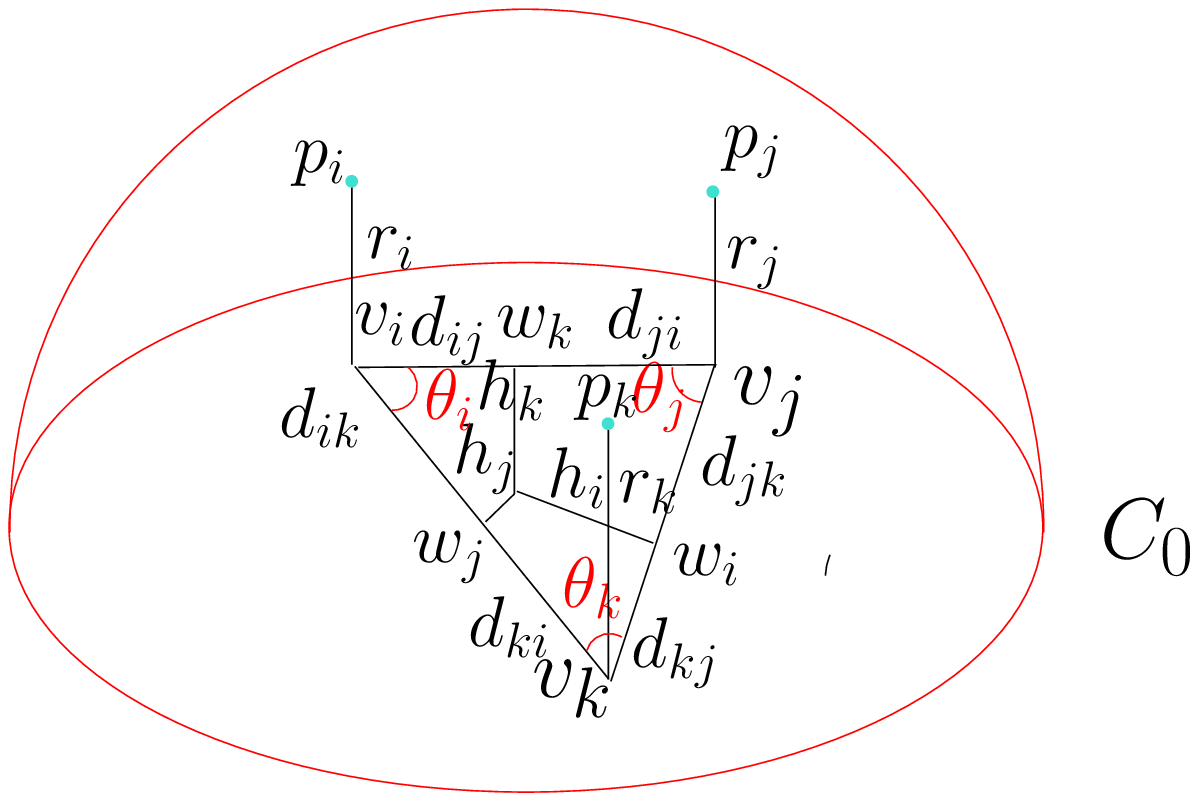}&
\includegraphics[width=0.23\textwidth]{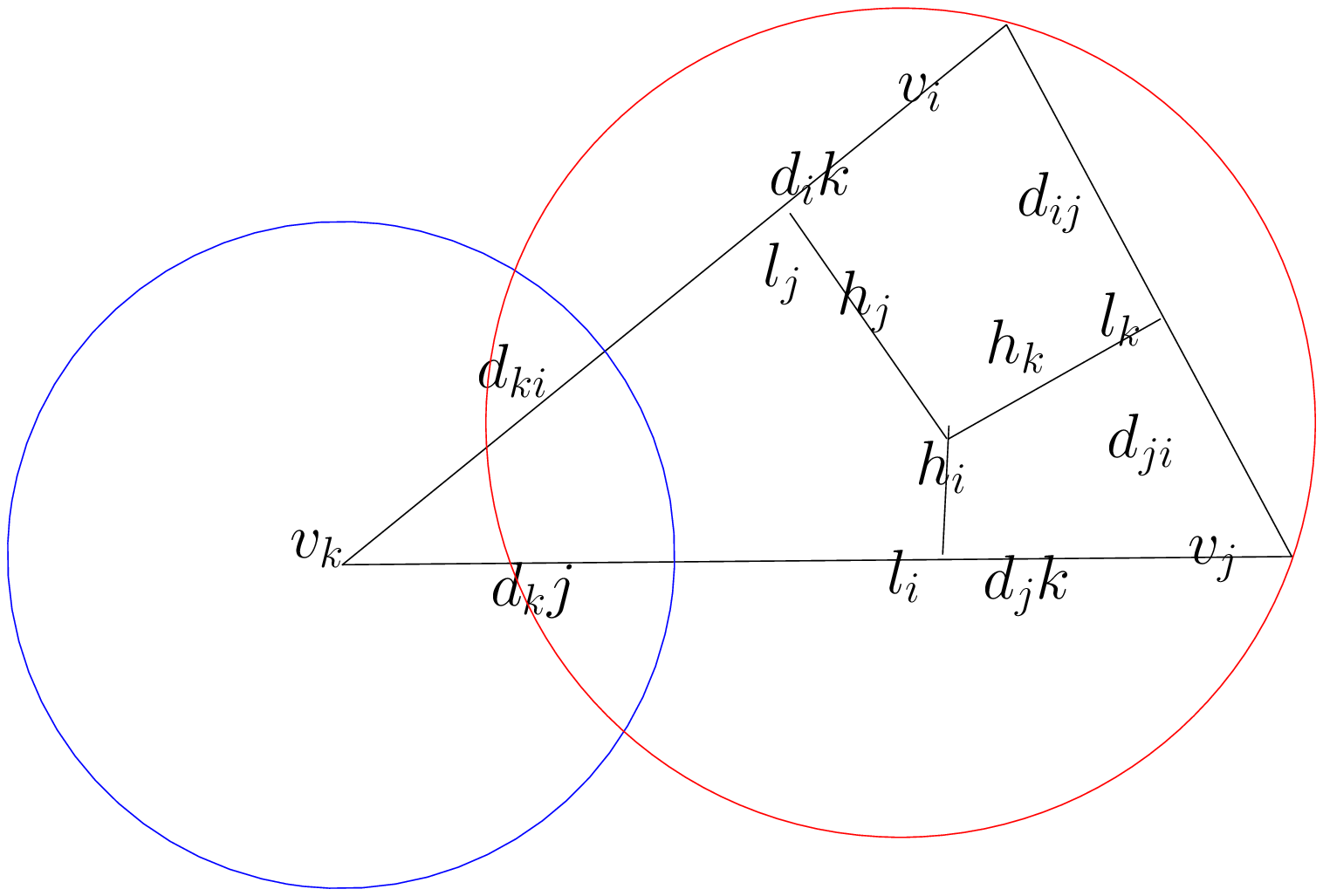}\\
(d) Yamabe flow & (e) virt.rad.cp & (f) mixed type\\
\includegraphics[width=0.2\textwidth]{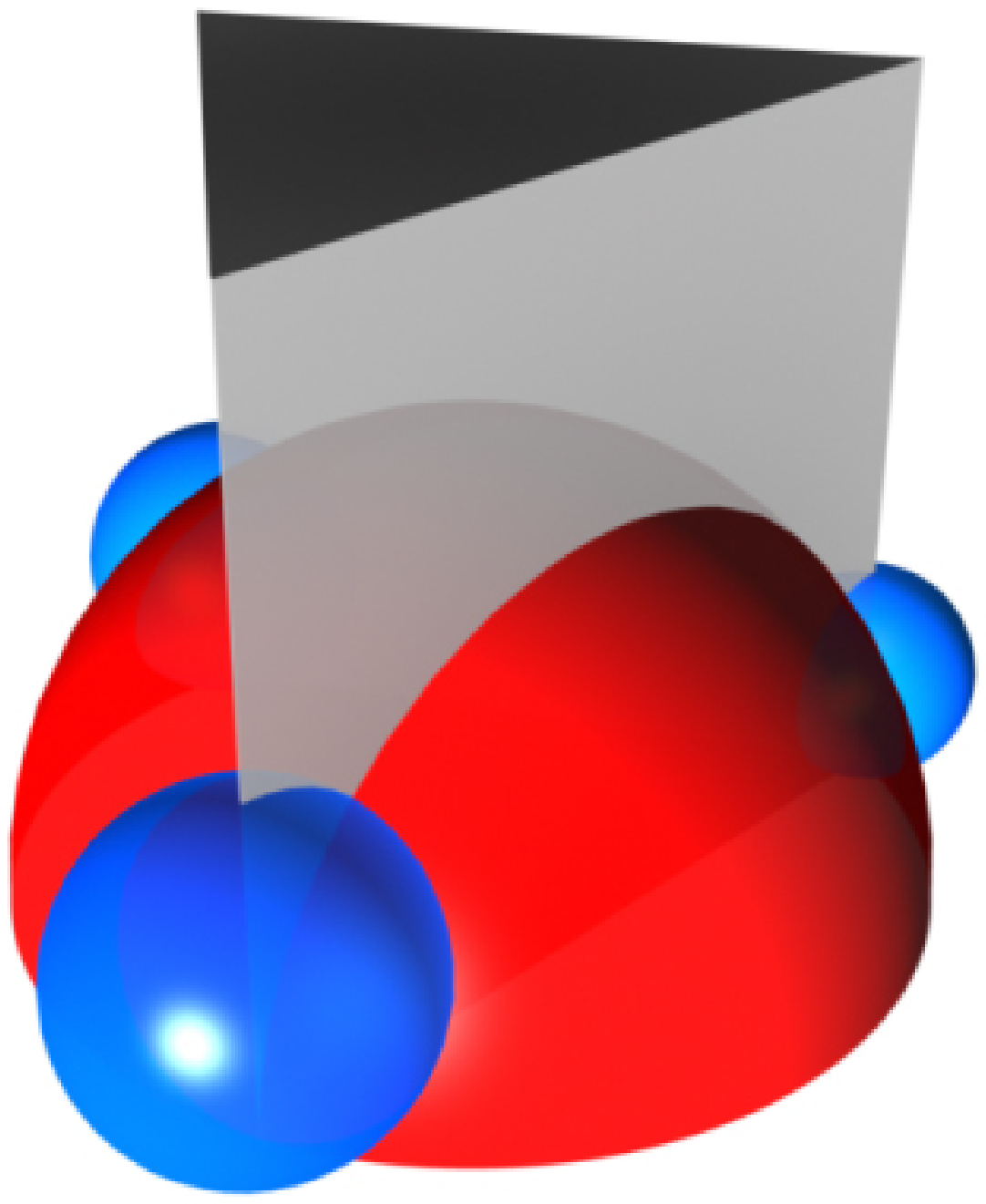}&
\includegraphics[width=0.2\textwidth]{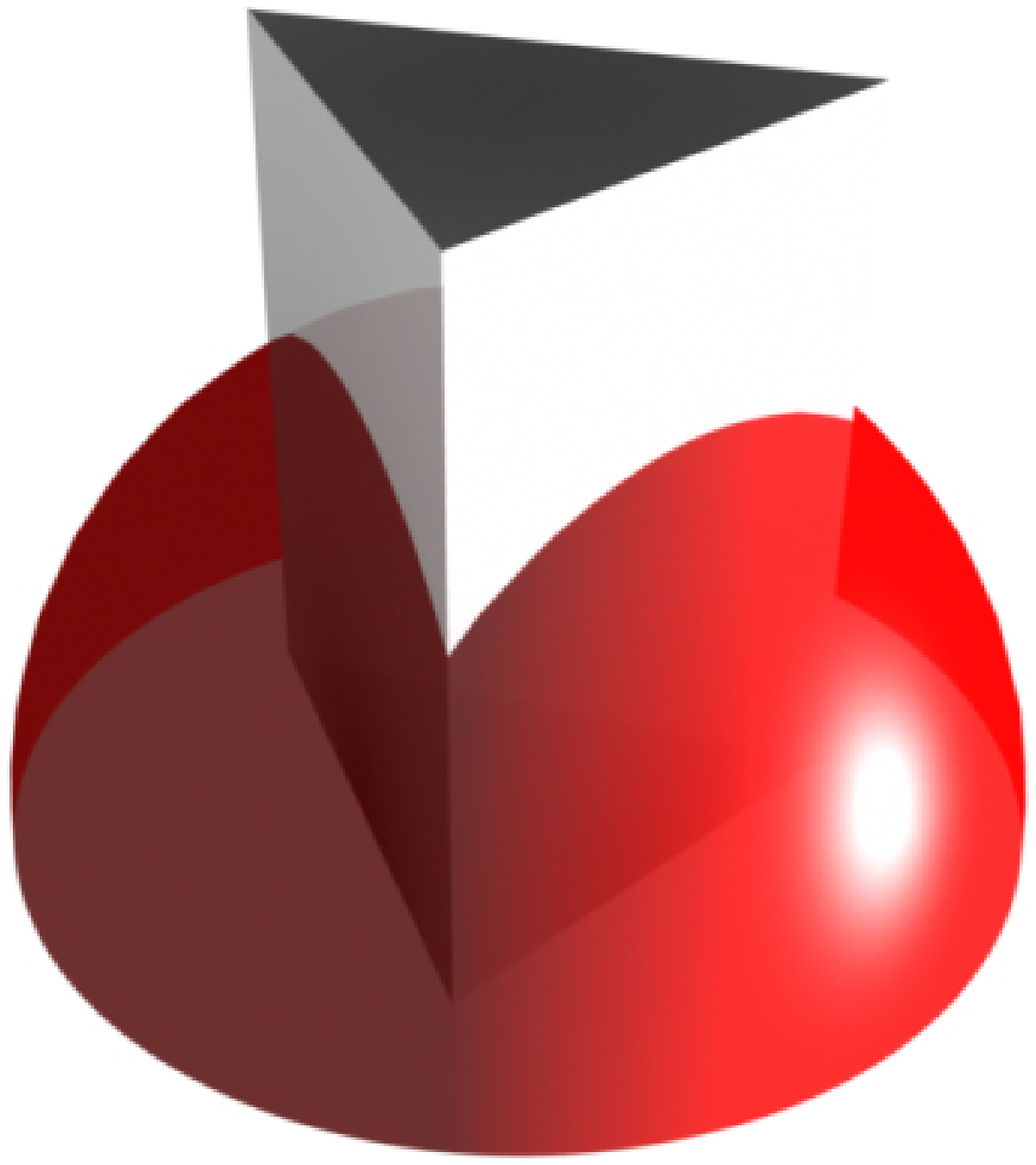}&
\includegraphics[width=0.2\textwidth]{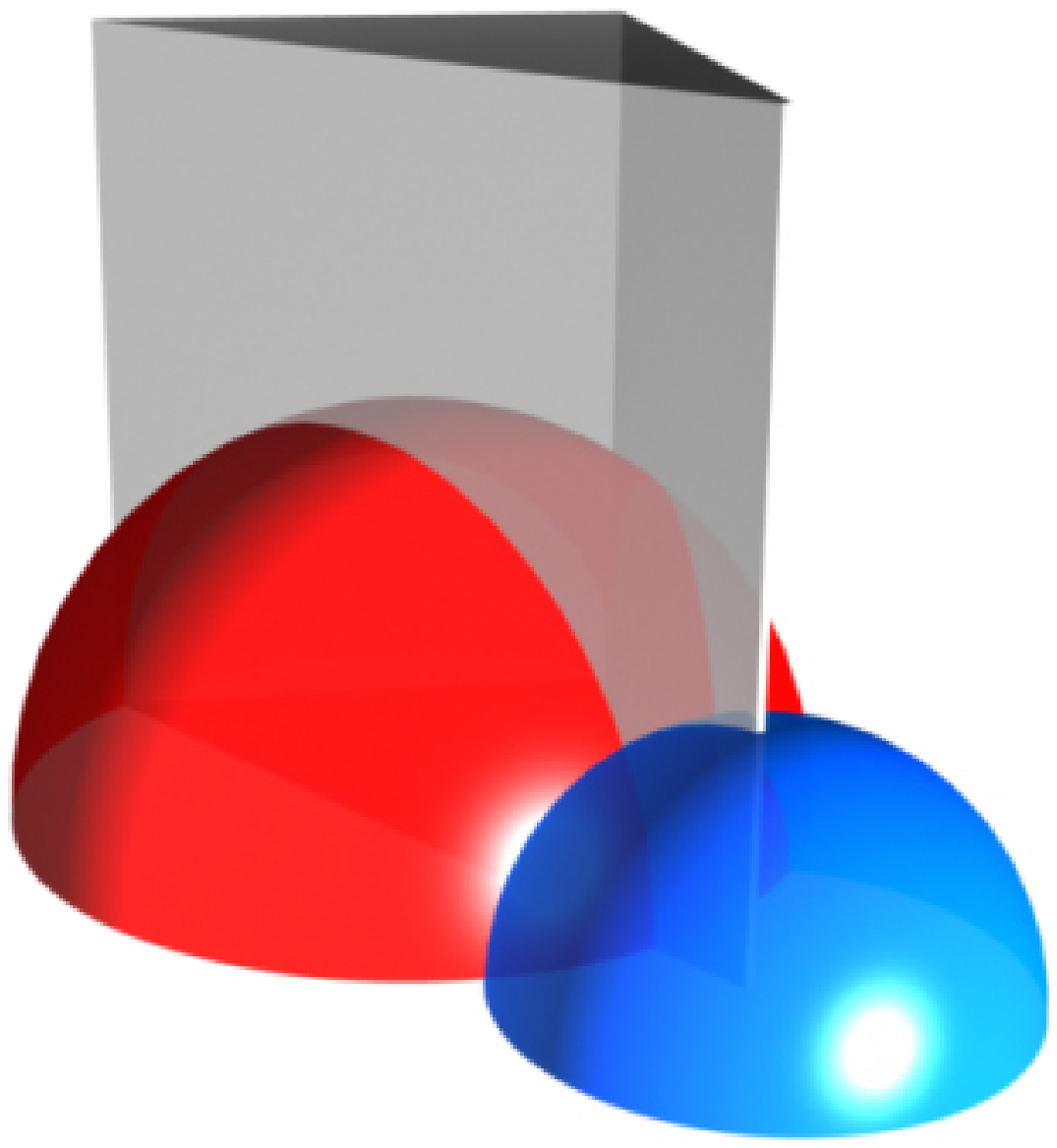}\\
$\eta>0,\epsilon=0$  & $\eta>0, \epsilon=-1$  & $\eta>0$\\
&&$\epsilon\in \{+1,0,-1\}$
\end{tabular}
\caption{Yamabe flow, virtual radius circle packing and mixed type
schemes, and the geometric interpretations to their Ricci energies.
\label{fig:schemes_2}}
\end{center}
\vspace{-0.5cm}
\end{figure}

\begin{definition}[Circle Packing Schemes]
Suppose $\Sigma=(V,E,F)$ is triangle mesh with spherical, Euclidean or hyperbolic background geometry. Given a circle packing metric $(\Sigma, \gamma, \eta, \epsilon)$, for an edge $[v_i,v_j]\in E$, its length $l_{ij}$ is given by
\begin{equation}
\left\{
\begin{array}{lclr}
l_{ij}^2 &=& 2\eta_{ij} e^{u_i+u_j} + \varepsilon_i e^{2u_i} + \varepsilon_j e^{2u_j}&\mathbb{E}^2\\
\cosh l_{ij} &=& \frac{4\eta_{ij}e^{u_i+u_j} + (1+\varepsilon_i e^{2u_i})(1+\varepsilon_j e^{2u_j})}{(1-\varepsilon_i e^{2u_i})(1-\varepsilon_j e^{2u_j})}&\mathbb{H}^2\\
\cos l_{ij} &=& \frac{-4\eta_{ij}e^{u_i+u_j} + (1-\varepsilon_i e^{2u_i})(1-\varepsilon_j e^{2u_j})}{(1+\varepsilon_i e^{2u_i})(1+\varepsilon_j e^{2u_j})}&\mathbb{S}^2\\
\end{array}
\right.
\label{eqn:unified_circle_packing_edge_lengths}
\end{equation}
The schemes are named as follows:
{\footnotesize \begin{center}
\begin{tabular}{|l|l|l|l|}
\hline
Scheme & $\varepsilon_i$ & $\varepsilon_j$ & $\eta_{ij}$\\
\hline
\hline
Tangential Circle Packing & +1 & +1 & +1 \\
\hline
Thurston's Circle Packing & +1 & +1 & $[0,1]$\\
\hline
Inversive Distance Circle Packing & +1 & +1 & $>0$ \\
\hline
Yamabe Flow & 0 & 0 & $>0$ \\
\hline
Virtual Radius Circle Packing & -1 & -1 & $>0$\\
\hline
Mixed type & $\{-1,0,+1\}$ & $\{-1,0,+1\}$ & $>0$\\
\hline
\end{tabular}
\end{center}}
\end{definition}

Fig.~\ref{fig:schemes_1} and Fig.~\ref{fig:schemes_2} illustrate all
the schemes with for discrete surfaces with Euclidean background
geometry.

\begin{remark}From the definition, the tangential circle packing is a special case of Thurston's circle packing; Thurston's circle packing is a special case of inversive distance circle packing. In the following discussion, we unify all three types as inversive distance circle packing.
\end{remark}

\subsection{Discrete Surface Ricci Flow}

\begin{definition}[Discrete Surface Ricci Flow] A discrete surface with $\mathbb{S}^2$, $\mathbb{E}^2$ or $\mathbb{H}^2$ background geometry, and a circle packing metric $(\Sigma,\gamma, \eta,\epsilon)$,
the discrete surface Ricci flow is
\begin{equation}
\frac{du_i(t)}{dt} = \bar{K}_i - K_i(t),
\label{eqn:discrete_ricci_flow}
\end{equation}
where $\bar{K}_i$ is the target curvature at the vertex $v_i$.
\end{definition}
The target curvature must satisfy certain constraints to ensure the existence of the solution to the flow, such as Gauss-Bonnet equation Eqn.~\ref{eqn:Gauss_Bonnet},  but also some additional ones described in \cite{Thurston97}, \cite{Marden_Rodin_1990} and \cite{chow_luo_03}, for instances.

The discrete surface Ricci flow has exactly the same formula as the
smooth counter part Eqn.~\ref{eqn:smooth_ricci_flow}. Furthermore,
similar to the smooth case, discrete surface Ricci flow is also
variational: the discrete Ricci flow is the negative gradient flow
of the discrete Ricci energy.

\begin{definition}[Discrete Ricci Energy]  A discrete surface with $\mathbb{S}^2$, $\mathbb{E}^2$ or $\mathbb{H}^2$ background geometry, and a circle packing metric $(\Sigma,\gamma, \eta,\epsilon)$. For a triangle $[v_i,v_j,v_k]$ with inner angles $(\theta_i,\theta_j,\theta_k)$, the discrete Ricci energy on the face is given by
\begin{equation}
 E_f( u_i,u_j, u_k) = \int^{(u_i,u_j,u_k)} \theta_i du_i + \theta_j du_j + \theta_k du_k.
 \label{eqn:face_ricci_energy}
\end{equation}
The discrete Ricci energy for the whole mesh is defined as
\begin{equation}
E_\Sigma(u_1,u_2,\cdots, u_n) = \int^{(u_1,u_2,\cdots, u_n)} \sum_{i=1}^n (\bar{K}_i-K_i) du_i.
 \label{eqn:mesh_ricci_energy}.
\end{equation}
\end{definition}
From definition, we get the relation between the surface Ricci energy and the face Ricci energy
\begin{equation}
    E_\Sigma = \sum_{i=1}^n (\bar{K}_i - 2\pi)u_i + \sum_{f\in F} E_f.
\label{eqn:relation}
\end{equation}
The description of the energy in terms of an integral requires the
fact that the inside is a closed form so that it is defined
independent of the integration path. This follows from the following
symmetry lemma, which has fundamental importance. In this work, we
give three proofs. The following one is algebraic, more difficult to
verify, but leads to computational algorithm directly. The second
one is based on the geometric interpretation to the Hessian matrix
in Section \ref{sec:Hessian}. The third one is based on the
geometric interpretation to the discrete Ricci energy. The later two
proofs are more geometric and intuitive.

\begin{lemma}[Symmetry]  A discrete surface with $\mathbb{S}^2$, $\mathbb{E}^2$ or $\mathbb{H}^2$ background geometry, and a circle packing metric $(\Sigma,\gamma, \eta,\epsilon)$, then for any pair of vertices $v_i$ and $v_j$:
\begin{equation}
    \frac{\partial K_i}{\partial u_j } = \frac{\partial K_j}{\partial u_i }.
    \label{eqn:symmetry}
\end{equation}
\label{lem:symmetry}
\end{lemma}

\begin{proof} From the relation in Eqn. \ref{eqn:relation}, it is sufficient and necessary to show the symmetry for each triangle $[v_i,v_j,v_k]$ for all schemes,
\[
    \frac{\partial \theta_i}{\partial u_j} = \frac{\partial \theta_j}{\partial u_i}.
\]
This is proven by finding the explicit formula for the Hessain matrix of the face Ricci energy,
\begin{equation}
\frac{\partial(\theta_i,\theta_j,\theta_k)}{\partial(u_i,u_j,u_k)} = -\frac{1}{2A} L\Theta L^{-1} D,
    \label{eqn:Hessain}
\end{equation}
where
\begin{equation}
    A = \frac{1}{2}\sin\theta_i s(l_j) s(l_k)
\end{equation}
the matrix $L$
\begin{equation}
    L = \left(
    \begin{array}{ccc}
    s(l_i)&0&0\\
    0 & s(l_j)& 0 \\
    0 & 0 & s(l_k)
    \end{array}
    \right)
\end{equation}
and the matrix $\Theta$
\begin{equation}
    \Theta = \left(
    \begin{array}{ccc}
    -1&\cos \theta_k & \cos \theta_j\\
    \cos\theta_k &-1 & \cos\theta_i\\
    \cos\theta_j & \cos\theta_i &-1
    \end{array}
    \right)
\end{equation}
and
\begin{equation}
    D = \left(
    \begin{array}{ccc}
    0&\tau(i,j,k) & \tau(i,k,j)\\
    \tau(j,i,k) &0 & \tau(j,k,i)\\
    \tau(k,i,j) & \tau(k,j,i) &0
    \end{array}
    \right)
\end{equation}
where $s(x)$ and $\tau(i,j,k)$ are defined as
\begin{center}
\begin{tabular}{|l|l|l|}
\hline
& $s(x)$ & $\tau(i,j,k)$ \\
\hline
\hline
$\mathbb{E}^2$ & $x$ & $1/2(l_i^2 + \varepsilon_j r_j^2 - \varepsilon_k r_k^2)$\\
\hline
$\mathbb{H}^2$ & $\sinh x$ & $\cosh l_i \cosh^{\varepsilon_j} r_j - \cosh^{\varepsilon_k}r_k$ \\
\hline
$\mathbb{S}^2$ & $\sin x$ & $\cos l_i \cos^{\varepsilon_j} r_j - \cos^{\varepsilon_k}r_k$ \\
\hline
\end{tabular}
\end{center}
By symbolic computation, it is straightforward to verify the symmetry of Eqn. \ref{eqn:Hessain}.
\vspace{-0.3cm}
\end{proof}

\section{Geometric interpretation to Hessian}
\label{sec:Hessian}

This section focuses on the geometric interpretation to Hessian
matrix of the discrete Ricci energy on each face for $\mathbb{E}^2,
\mathbb{H}^2$ and $\mathbb{S}^2$ cases. This gives the second proof
of the symmetry lemma \ref{lem:symmetry}.

\subsection{Euclidean Case}

The interpretation in Euclidean case is due to
Glickenstein~\cite{Glickenstein_2011} (Z. He~\cite{ZHe99} in the
case of circle packings) and illustrated in \cite{ricci2013}. In the
current work, we build the connection to the Power Delaunay
triangulation and power voronoi diagram.

We only focus on one triangle $[v_i,v_j,v_k]$, with corner angles $\theta_i,\theta_j,\theta_k$, conformal factors $u_i,u_j,u_k$ and edge lengths $l_{ij}$ for edge $[v_i,v_j]$, $l_{jk}$ for $[v_j,v_k]$ and $l_{ki}$ for $[v_k,v_i]$.

\paragraph{Power Delaunay Triangulation}

As shown in Fig. \ref{fig:schemes_1} and Fig.~\ref{fig:schemes_2},
the \emph{power} of $q$ with respect to $v_i$ is
\[
    pow( v_i, q) = | v_i - q |^2 - \epsilon \gamma_i^2.
\]
The \emph{power center} $o$ of the triangle satisifies
\[
    pow(v_i,o) = pow(v_j,o) = pow(v_k,o).
\]
The \emph{power circle} $C$ centered at $o$ with radius $\gamma$, where $\gamma = pow(v_i,o)$.

Therefore, for tangential, Thurton's and inversive distance circle packing cases, the power circle is orthogonal to  three circles at the vertices $C_i$, $C_j$ and $C_k$; for Yamabe flow case, the power circle is the circumcircle of the triangle; for virtual radius circle packing, the power circle is the equator of the sphere, which goes through three points $\{v_i + \gamma_i^2 \mathbf{n}, v_j + \gamma_j^2 \mathbf{n}, v_k+\gamma_k^2\mathbf{n}\}$, where $\mathbf{n}$ is the normal to the plane.

Through the power center, we draw line perpendicular to three edges, the perpendicular feets are $w_i,w_j$ and $w_k$ respectively. The distance from the power center to the perpendicular feet are $h_i,h_j$ and $h_k$ respectively. Then it can be shown easily that
\begin{equation}
    \frac{\partial \theta_i}{\partial u_j} = \frac{\partial \theta_j}{\partial u_i} = \frac{h_k}{l_k},
    \frac{\partial \theta_j}{\partial u_k} = \frac{\partial \theta_k}{\partial u_j} = \frac{h_i}{l_i},
    \frac{\partial \theta_k}{\partial u_i} = \frac{\partial \theta_i}{\partial u_k} = \frac{h_j}{l_j},
    \label{eqn:geometric_hessian_1}
\end{equation}
furthermore,
\begin{equation}
\frac{\partial \theta_i}{\partial u_i} = -\frac{h_k}{l_k} - \frac{h_j}{l_j},
\frac{\partial \theta_j}{\partial u_j} = -\frac{h_k}{l_k} - \frac{h_i}{l_i},
\frac{\partial \theta_k}{\partial u_k} = -\frac{h_i}{l_i} - \frac{h_j}{l_j}.
    \label{eqn:geometric_hessian_2}
\end{equation}
These two formula induces the formula for the Hessian of the Ricci energy of the whole surface. One can treat the circle packing $(\Sigma, \gamma, \eta, \varepsilon)$  as a power triangulation, which has a dual power diagram $\bar{\Sigma}$. Each edge $e_{ij}\in \Sigma$ has a dual edge $\bar{e}\in \bar{\Sigma}$, then
\begin{equation}
\frac{\partial K_i}{\partial u_j} = \frac{\partial K_j}{\partial
u_i} = \frac{|\bar{e}_{ij}|}{|e_{ij}|},
\label{eqn:Euclidean_Hessian_1}
\end{equation}
and
\begin{equation}
\frac{\partial K_i}{\partial u_i} = -\sum_j \frac{\partial
K_i}{\partial u_j}. \label{eqn:Euclidean_Hessian_2}
\end{equation}
This gives a geometric proof for the symmetry lemma
\ref{lem:symmetry} in Euclidean case.

Suppose on the edge $[v_i,v_j]$, the distance from $v_i$ to the perpendicular foot $w_k$ is $d_{ij}$, the distance from $v_j$ to $w_k$ is $d_{ji}$, then $l_{ij} = d_{ij}+d_{ji}$, and
\[
    \frac{\partial l_{ij}}{\partial u_i} = d_{ij},     \frac{\partial l_{ij}}{\partial u_j} = d_{ji},
\]
furthermore
\[
    d_{ij}^2+d_{jk}^2+d_{ki}^2 = d_{ik}^2+d_{kj}^2+ d_{ji}^2.
\]
This shows the power circle interpretation is equivalent to Glikenstain's formulation.

\subsection{Hyperbolic Case}
Let $\triangle_{123}$ be a hyperbolic triangle whose vertices are labeled by $1,2,3.$ Let $r_1,r_2,r_3$ be three positive numbers associated to the vertices, and $\epsilon_1,\epsilon_2,\epsilon_3\in \{-1,0,1\}$ be indicators of the type of the vertices.

For the mixed type of discrete conformal geometry, the edge length
of $\triangle_{123}$ is given by
{\[
\cosh l_k=4\eta_{ij} \frac{\sinh r_i}{(1-\epsilon_i)\cosh
r_i+1+\epsilon_i}  \frac{\sinh r_j}{(1-\epsilon_j)\cosh
r_j+1+\epsilon_j}
+ \cosh^{\epsilon_i} r_i \cosh^{\epsilon_j} r_j,
\]}

 where $\{i,j,k\}={1,2,3}.$

Via the cosine law, the edge lengths $l_1,l_2,l_3$ determine the angles $\theta_1,\theta_2,\theta_3$.

When $\epsilon_1=\epsilon_2=\epsilon_3=0$, this is the case of Yamabe flow. There is a circle passing through the three vertices of $\triangle_{123}$. It is still called the {\it power circle}.

When $\epsilon_1=\epsilon_2=\epsilon_3=1$, this is the case of inversive distance circle packing. Centered at each vertex $i$, there is a circle with radius $r_i$. Then there is the {\it power circle} orthogonal to the three circles centered at the vertices.

When $\epsilon_1=\epsilon_2=\epsilon_3=-1$, this is the case of virtual radius circle packing. Let $\triangle_{123}$ be on the equator plane of the ball model of the hyperbolic space $\mathbb{H}^3$. For each vertex $i$, let $ii'$ be the geodesic arc perpendicular to the equator plane with length $r_i$. Assume $1',2',3'$ are above the equator plane. There is a hemisphere passing through $1',2',3'$ and orthogonal to the equator plane. The {\it power circle} in this case is the intersection of the hemisphere and the equator plane.

For a mix type, the power circle can still be defined.

For any type, let $h_i$ be the distance from the center of the power circle to the edge $ij$ whose length is $l_k$.


\begin{theorem}
Let $$e^{u_i}=\frac{e^{r_i}-1}{e^{r_i}+1}=\tanh \frac{r_i}2.$$ Then
\[
    \frac{\partial \theta_1}{\partial u_2}=\frac{\partial \theta_2}{\partial u_1}
\]
which equal to
{\small
\[
\frac{\tanh h_3}{\sinh^2 l_3}\sqrt{2\cosh^{\epsilon_1}r_1 \cosh^{\epsilon_2}r_2 \cosh l_3
-\cosh^{2\epsilon_1}r_1-\cosh^{2\epsilon_2}r_2}.
\]
} \label{thm:hyperbolic_hessain}
\end{theorem}
This gives a geometric proof for the symmetry lemma \ref{lem:symmetry} in hyperbolic case.

We only need to prove the theorem for the case of $\epsilon_1=\epsilon_2=\epsilon_3=1$. General case can be proved similarly.

\begin{proof}
\textbf{Step 1}. Denote the center of the power circle by $o$, the radius by $r$. Let $x,y,z$ be the distance from $o$ to the vertices $1,2,3$.
Then
\begin{equation}\label{equ:xyz}
\begin{array}{lcl}
\cosh x&=&\cosh r \cosh r_1\\
\cosh y&=&\cosh r \cosh r_2\\
\cosh z&=&\cosh r \cosh r_3
\end{array}
\end{equation}

Let $\alpha$ be the angle $\angle 13o$ and $\beta$ the angle $\angle 23o$. Then $\alpha+\beta=\theta_3.$ Therefore
\begin{equation}\label{equ:3angle}
1+2\cos \alpha \cos \beta \cos \theta_3 = \cos^2\alpha + \cos^2 \beta +\cos^2 \theta_3.
\end{equation}

By the cosine law,
$$\cos \alpha=\frac{-\cosh x+ \cosh z \cosh l_2}{\sinh z \sinh l_2},$$
$$\cos \beta=\frac{-\cosh y+ \cosh z \cosh l_1}{\sinh z \sinh l_1},$$
$$\cos \theta_3=\frac{-\cosh l_3+ \cosh l_1 \cosh l_2}{\sinh l_1 \sinh l_2}.$$

Substituting the three formulas into the equation (\ref{equ:3angle}), we obtain a relation between the 6 numbers $l_1,l_2,l_3, x,y,z$.

Substituting the equations (\ref{equ:xyz}) into this relation, we obtain a relation between $l_1,l_2,l_3, r_1,r_2,r_3$ and $r$.

Solving for $r$, we get $\cosh^2 r=\frac{\mathcal{N}}{\mathcal{D}},$ where
{
$$\mathcal{N}=
1+2\cosh l_1 \cosh l_2 \cosh l_3 - \cosh^2 l_1 - \cosh^2 l_2 - \cosh^2 l_3,$$
}
{
\begin{align*}
\mathcal{D}&=\cosh^2 r_1 (1-\cosh^2 l_1) + 2 \cosh r_2 \cosh r_3 (\cosh l_2 \cosh l_3-\cosh l_1)\\
           &+\cosh^2 r_2 (1-\cosh^2 l_2) + 2 \cosh r_3 \cosh r_1 (\cosh l_3 \cosh l_1-\cosh l_2)\\
           &+\cosh^2 r_3 (1-\cosh^2 l_3) + 2 \cosh r_1 \cosh r_2 (\cosh l_1 \cosh l_2-\cosh l_3).
\end{align*}
}
\textbf{Step 2}. Since $h_3$ is the height of the triangle $\triangle_{o12}$ with bottom the edge $12$. By the standard formula of height of a hyperbolic triangle, we have
{\small
\[
\sinh^2 h_3=\frac{1+2\cosh x \cosh y \cosh l_3 - \cosh^2 x - \cosh^2 y - \cosh^2 l_3}{\sinh^2 l_3}.
\]}

After substituting the equations (\ref{equ:xyz}) into the above formula, we have
{\[
\tanh^2 h_3=\frac{\cosh^2 r(2\cosh r_1 \cosh r_2 \cosh l_3-\cosh^2 r_1 - \cosh^2 r_2 )- \sinh^2 l_3}
{\cosh^2 r(2\cosh r_1 \cosh r_2 \cosh l_3-\cosh^2 r_1 - \cosh^2 r_2 )}.
\]}
After substituting the equation $\cosh^2 r=\frac{\mathcal{N}}{\mathcal{D}},$ we have
{\[
\tanh^2 h_3=\frac{\mathcal{N}(2\cosh r_1 \cosh r_2 \cosh l_3-\cosh^2 r_1 - \cosh^2 r_2 )-\mathcal{D} \sinh^2 l_3}
{\mathcal{N}(2\cosh r_1 \cosh r_2 \cosh l_3-\cosh^2 r_1 - \cosh^2 r_2 )}.
\]}
After substituting the expressions of $\mathcal{N}$ and $\mathcal{D}$ in step 1, we have
{
\begin{align*}
&\tanh^2 h_3=\\
&\frac{[(\cosh l_1 \cosh l_3-\cosh l_2)\cosh r_1+ (\cosh l_2 \cosh l_3-\cosh l_1)\cosh r_2-\sinh^2 l_3 \cosh r_3]^2}
{\mathcal{N}(2\cosh r_1 \cosh r_2 \cosh l_3-\cosh^2 r_1 - \cosh^2 r_2 )}.
\end{align*}}

\textbf{Step 3}. By direct calculation, we have
{\begin{align*}
&\frac{\partial \theta_1}{\partial u_2}=\frac{\partial \theta_2}{\partial u_1}=\frac{-1}{\sin \theta_i \sinh l_j \sinh l_k}\\
&(\cosh r_3-\frac{\cosh l_1 \cosh l_3-\cosh l_2}{\sinh^2 l_3}\cosh r_1-\frac{\cosh l_2 \cosh l_3-\cosh l_1}{\sinh^2 l_3}\cosh r_2)=\\
&\frac {(\cosh l_1 \cosh l_3-\cosh l_2)\cosh r_1+(\cosh l_2 \cosh l_3-\cosh l_1)\cosh r_2-\sinh^2 l_3\cosh r_3}
{\sqrt{\mathcal{N}}\cdot \sinh^2 l_3}.
\end{align*}}
Comparing with the last formula of step 2, we have
{ $$\frac{\partial \theta_1}{\partial u_2}=\frac{\partial \theta_2}{\partial u_1}=\frac{\tanh h_3}{\sinh^2 l_3}\sqrt{2\cosh r_1 \cosh r_2 \cosh l_3-\cosh^2 r_1 - \cosh^2 r_2}.$$}
\end{proof}


\subsection{Spherical Case}

According to a general principle of the relation of hyperbolic
geometry and spherical geometry, to obtain a formula in spherical
geometry, we only need to replace $\sinh$ and $\cosh$ in hyperbolic
geometry by $\sqrt{-1}\sin$ and $\cos.$

For the mixed type of discrete conformal geometry with spherical
background geometry, the edge length of $\triangle_{123}$ is given by
{
\begin{align*}
&\cosh l_{ij}=\\
&-4\eta_{ij} \frac{\sin r_i}{(1-\epsilon_i)\cos
r_i+1+\epsilon_i}  \frac{\sin r_j}{(1-\epsilon_j)\cos
r_j+1+\epsilon_j} + \cos^{\epsilon_i} r_i \cos^{\epsilon_j} r_j.
\end{align*}}

Via the cosine law, the edge lengths $l_1,l_2,l_3$ determine the
angles $\theta_1,\theta_2,\theta_3$.

We can define power circles similarly. Let $h_i$ be the distance from
the center of the power circle to the edge $ij$ whose length is $l_k$.


\begin{theorem}
Let $$e^{u_i}=\tan \frac{r_i}2.$$ Then
\[
\frac{\partial \theta_1}{\partial u_2}=\frac{\partial
\theta_2}{\partial u_1}
\]
which equal to
{\small\[
\frac{\tan h_3}{\sin^2 l_3}\sqrt{-2\cos^{\epsilon_1}r_1
\cos^{\epsilon_2}r_2 \cos l_3
+\cos^{2\epsilon_1}r_1+\cos^{2\epsilon_2}r_2}.
\]}
\label{thm:spherical_hessain}
\end{theorem}
This gives a geometric proof for the symmetry lemma \ref{lem:symmetry} in spherical case.

This theorem is also proved by using the general principle: replace
$\sinh$ and $\cosh$ in hyperbolic geometry by $\sqrt{-1}\sin$ and
$\cos.$

Here we can give the second proof for the symmetry lemma
\ref{lem:symmetry} based on the geometric interpretation to the
Hessian, which is geometric and intuitive.

\begin{proof}
Formula \ref{eqn:geometric_hessian_1} show the symmetry for all schemes with
Euclidean background geometry; theorem~\ref{thm:hyperbolic_hessain}
proves the symmetry for the hyperbolic cases;
theorem~\ref{thm:spherical_hessain} for the spherical cases.
\end{proof}

\section{Geometric Interpretations to Ricci Energies}
\label{sec:energy}
The geometric interpretation to Ricci energies of
Euclidean and hyperbolic Yamabe schemes were discovered by Bobenko,
Pinkall and Springborn in \cite{Bobenko_Pinkall_Springborn_2010}.
The interpretation to Ricci energies of Euclidean schemes (without the mixed type) are illustrated in
\cite{ricci2013}. In the current work, we generalize the geometric interpretations to all
the schemes in all background geometries covered by the unified framework, as shown in Fig.~\ref{fig:energy_volume}.


\begin{figure}
\begin{center}
\begin{tabular}{lc}
\includegraphics[width=0.22\textwidth]{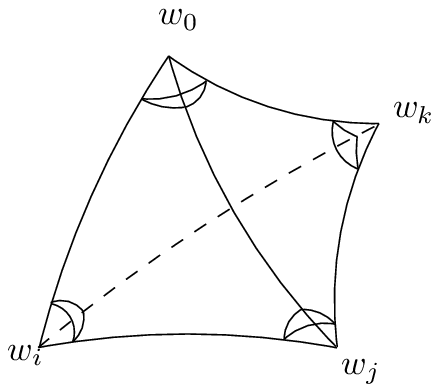}&
\includegraphics[width=0.27\textwidth]{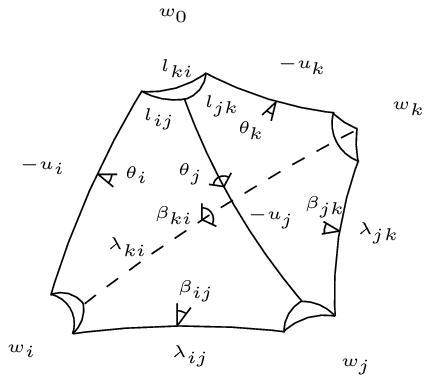}
\end{tabular}
\caption{Generalized hyperbolic tetrahedron.
\label{fig:hyperbolic_tetrahedron}}
\end{center}
\end{figure}

We use the upper half space model for $\mathbb{H}^3$, with Riemannian metric
\[
    ds^2 = \frac{dx^2+dy^2+dz^2}{z^2}
\]
the xy-plane is the ideal boundary. Consider a triangle $[v_i,v_j,v_k]$, its Ricci energy is closely related to the volume of a generalized hyperbolic tetrahedron whose vertices can be in $\mathbb{H}^3$, truncated by a horosphere or truncated by a hyperbolic plane.

In Fig. ~\ref{fig:hyperbolic_tetrahedron}, the generalized hyperbolic tetrahedron has 4 vertices $w_0, w_i, w_j, w_k.$ The tetrahedron vertex $w_0$ is called the \emph{top vertex}. The $4$ faces of the tetrahedron are hyperbolic planes, the $6$ edges are geodesics. The $6$ edge lengths of the generalized tetrahedron are $-u_i,-u_j,-u_k$ and $\lambda_{ij},\lambda_{jk}, \lambda_{ki}$. The generalized tetrahedron is uniquely determined by these $6$ edge lengths.

The followings are the common principles for constructing the generalized tetrahedron for all the schemes,
\begin{enumerate}
\item For all $\mathbb{E}^2$ schemes, the top vertex $w_0$ is ideal (at infinity) and truncated by a horosphere; for all $\mathbb{H}^2$ schemes, the top vertex is hyperideal (exceeding the boundary of $\mathbb{H}^3$) and truncated by a hyperbolic plane; for all $\mathbb{S}^2$ schemes, the top vertex is in $\mathbb{H}^3$.
\item For $w_i$, if the corresponding vertex $v_i$ is of inversive distance circle packing $\varepsilon_i=+1$, then it is hyperideal and truncated by a hyperbolic plane; if $v_i$ is of Yamabe flow $\varepsilon_i=0$, then it is ideal and truncated by a horosphere; if $v_i$ is virtual radius circle packing $\varepsilon_1=-1$, then it is in $\mathbb{H}^3$. Same results holds for $w_j$ and $w_k$.
\item The edges on the truncated tetrahedron, connecting to the top vertex on the original tetrahedron, have lengths $-u_i$, $-u_j$ and $-u_k$ respectively.
\item For the edge lengths $\lambda_{ij}$, there is a unified formula for three geometries: Euclidean, hyperbolic, spherical,
\begin{equation}
    \eta_{ij} = \frac{1}{2}( e^{\lambda_{ij}} + \varepsilon_i \varepsilon_j e^{-\lambda_{ij}}).
    \label{eqn:bottom_edge_length}
\end{equation}
\end{enumerate}

The triangle associated to the top vertex $w_0$ is the triangle $[v_i,v_j,v_k]$. It is obtained by truncating by a horosphere, truncating by a hyperbolic plane or intersecting with a sphere. Given $-u_i, -u_j, -u_k, \eta_{ij}, \eta_{jk}, \eta_{ki}$, using cosine law, we can calculate the edge lengths of the triangle $[v_i,v_j,v_k]$. They are exactly given by the formula Eqn.~\ref{eqn:unified_circle_packing_edge_lengths}. That means the triangle $[v_i,v_j,v_k]$ has lengths $l_{ij}, l_{jk}, l_{ki}$ and angles $\theta_i, \theta_j, \theta_k$.

Here we can give the third proof for the symmetry lemma based on the
geometric interpretation to the Ricci energy, which is more
geometric, intuitive and much easier to verify.
\begin{proof}
As shown in Fig. ~\ref{fig:hyperbolic_tetrahedron}, for a generalized hyperbolic tetrahedron, the $4$ vertices can have any types. The $3$ vertical edges have lengths $-u_i,-u_j,-u_k$ with dihedral angles $\theta_i,\theta_j,\theta_k$. The bottom edges have lengths $\lambda_{ij}, \lambda_{jk}, \lambda_{ki}$ with dihedral angles $\beta_{ij}, \beta_{jk},\beta_{ki}$.

Let $V$ be the volume of the generalized hyperbolic tetrahedron. By Schl\"{a}fli formula
{\begin{equation}
dV = -\frac{1}{2}\left( -u_id\theta_i - u_jd\theta_j - u_kd\theta_k + \lambda_{ij} d\beta_{ij} + \lambda_{jk} d\beta_{jk} + \lambda_{ki} d\beta_{ki} \right)
\label{eqn:schlafli}
\end{equation}}
During the Ricci flow, the conformal structure coefficients $\eta_{ij}, \eta_{jk},\eta_{ki}$ are invariant, so $\lambda_{ij}, \lambda_{jk},\lambda_{ki}$ are fixed. Because the generalized tetrahedron is determined by the edge lengths $-u_i,-u_j,-u_k,\lambda_{ij}, \lambda_{jk},\lambda_{ki}$, during the flow, all dihedral angles $\theta_i,\theta_j,\theta_k, \beta_{ij},\beta_{jk},\beta_{ki}$ are functions of $u_i,u_j,u_k$, the volume $V$ is also the function of $u_i,u_j,u_k$.

Consider the function,
{
\begin{equation}
W(u_i,u_j,u_k)=u_i\theta_i + u_j\theta_j + u_k\theta_k - \lambda_{ij}\beta_{ij} - \lambda_{jk}\beta_{jk}-\lambda_{ki}\beta_{ki} - 2V
\label{eqn:geometric_ricci_flow_energy}
\end{equation}}
hence,
{\footnotesize
\[
\begin{array}{lcl}
dW &=& \theta_i du_i + \theta_j du_j + \theta_k du_k \\
&& + u_i d\theta_i + u_j d\theta_j + u_k d\theta_k -\lambda_{ij} d\beta_{ij} - \lambda_{jk}d\beta_{jk} - \lambda_{ki} d\theta_{ki}\\
&& -2dV
\end{array}
\]}
substitute Schl\"{a}fli formula Eqn.~\ref{eqn:schlafli}, we have
\[
    dW = \theta_i du_i + \theta_j du_j + \theta_k du_k
\]
therefore
\[
    W = \int \theta_i du_i + \theta_j du_j + \theta_k du_k + c.
\]
$W$ in fact, is the discrete Ricci energy on face in Eqn.~\ref{eqn:face_ricci_energy}. This shows the differential 1-form
\begin{equation}
    \theta_i du_i + \theta_j du_j + \theta_k du_k
    \label{eqn:differential}
\end{equation}
is exact, therefore closed. Namely, the Hessian matrix
\[
    \frac{\partial (\theta_i,\theta_j,\theta_k)}{\partial (u_i,u_j,u_k)}
\]
is symmetric.
\end{proof}
The formula Eqn.~\ref{eqn:geometric_ricci_flow_energy}
represents the Ricci energy on a face as the volume of the
generalized hyperbolic tetrahedron with other terms of conformal
factors and conformal structure coefficients. This formula was
introduced first by Bobenko, Pinkall and Springborn in
\cite{Bobenko_Pinkall_Springborn_2010} for Euclidean and hyperbolic
Yamabe flow. In the current work, we generalize it to all $18$
schemes. The differential in Eqn.~\ref{eqn:differential} is independent of the choice of horospheres, since the Schl\"{a}fli formula is independent of the choice of horospher for an ideal vertex.

\section{Conclusion}
\label{sec:conclusion}

This work establishes a unified framework for discrete surface Ricci
flow, which covers most existing schemes: tangential circle packing,
Thurston's circle packing, inversive distance circle packing,
discrete Yamabe flow, virtual radius circle packing and mixed
scheme, with Spherical, Euclidean and hyperbolic background
geometry. The unified frameworks for hyperbolic and spherical
schemes are introduced to the literature for the first time. For
Euclidean schemes, our formulation is equivalent to Glickenstein's
geometric construction.

Four newly discovered schemes are introduced, which are hyperbolic
and Euclidean virtual radius circle packing and the mixed schemes.

This work introduces a geometric interpretation to the Hessian of
discrete Ricci energy for all schemes, which generalizes
Glickenstein's formulation in Euclidean case.

This work also gives explicit geometric interpretations to the
discrete Ricci energy for all the schemes, which generalizes
Bobenko, Pinkall and Springborn's construction
\cite{Bobenko_Pinkall_Springborn_2010} for Yamabe flow cases.

The unified frame work deepen our understanding to the the discrete
surface Ricci flow theory, and inspired us to discover the novel
schemes of virtual radius circle packing and the mixed scheme,
improved the flexibility and robustness of the algorithms, greatly
simplified the implementation and improved the efficiency.

In the future, we will focus on answering the following open
problems: whether all possible discrete surface Ricci flow schemes
are the variations of the current unified approach on the primal
meshes and the dual diagrams and so on.
\vspace{-0.6cm} 
\bibliographystyle{abbrv}
\bibliography{../bib/chapter_2,../bib/chapter_3,../bib/chapter_4,../bib/chapter_5}

\section*{Appendix}
\label{sec:appendix}
In the appendix, we explain the unified surface Ricci flow algorithm \ref{alg:unified_ricci_flow} in details, and reorganize all the formulae necessary for the coding purpose.

\begin{algorithm}
\caption{
\footnotesize{Unified Surface Ricci Flow}\label{alg:unified_ricci_flow}}
\footnotesize{\begin{algorithmic}[1]
\REQUIRE The inputs include:\\
1. A triangular mesh $\Sigma$, embedded in $\mathbb{E}^3$; \\
2. The background geometry, $\mathbb{E}^2$, $\mathbb{H}^2$ or $\mathbb{S}^2$;\\
3. The circle packing scheme, $\epsilon\in \{+1,0,-1\}$;\\
4. A target curvature $\bar{K}$, $\sum \bar{K}_i = 2\pi
\chi(\Sigma)$ and $\bar{K}_i \in (-\infty, 2\pi)$.\\
5. Step length $\delta t$

\ENSURE A discrete metric conformal to the original one, which
realizes the target curvature $\bar{K}$. \STATE Initialize the
circle radii $\gamma$, discrete conformal factor $u$ and conformal
structure coefficient $\eta$, obtain the initial circle packing
metric $(\Sigma,\gamma, \eta, \epsilon)$ \WHILE{ $\max_i |\bar{K}_i
- K_i|
> threshold $} \STATE Compute the circle radii $\gamma$ from the
conformal factor $u$ \STATE Compute the edge length from $\gamma$
and $\eta$ \STATE Compute the corner angle $\theta_i^{jk}$ from the
edge length using cosine law \STATE Compute the vertex curvature $K$
\STATE Compute the Hessian matrix $H$ \STATE Solve linear system
$H\delta u = \bar{K}- K$ \STATE Update conformal factor $u
\leftarrow u - \delta t \times \delta u$ \ENDWHILE \STATE Output the
result circle packing metric.
\end{algorithmic}}
\end{algorithm}

\paragraph{ Step 1. Initial Circle Packing $(\gamma,\eta)$} Depending on different schemes, the initialization of the circle packing is different. The mesh has induced Euclidean metric $l_{ij}$. For inversive distance circle packing, we choose
\[
    \gamma_i = \frac{1}{3} \min_j  l_{ij},
\]
this ensures all the vertex circles are separated. For Yamabe flow, we choose all $\gamma_i$ to be $1$. For virtual radius circle packing, we choose all $\gamma_i$'s to be $1$. Then $\gamma_{ij}$ can be computed using the $l_{ij}$ formula in Tab. \ref{tab:summary}.
\paragraph{ Step 3. Circle Radii $\gamma$} The computation for circle radii from conformal factor uses the formulae in the first column in Tab.\ref{tab:summary}.
\paragraph{ Step 4. Edge Length $l$} The computation of edge lengths from conformal factor $u$ and conformal structure coefficient $\eta$ uses the formulae in the 2nd column in Tab.\ref{tab:summary}
\paragraph{ Step 5. Corner Angle $\theta$} The computation from edge length $l$ to the corner angle $\theta$ uses the cosine law formulae,
\[
    \begin{array}{rclr}
    l_k^2 &=& \gamma_i^2 + \gamma_j^2 - 2l_il_j \cos\theta_k & \mathbb{E}^2\\
    \cosh l_k &=& \cosh l_i \cosh l_j - \sinh l_i \sinh l_j \cos\theta_k &\mathbb{H}^2\\
    \cos l_k &=& \cos l_i \cos l_j - \sin l_i \sin l_j \cos\theta_k &\mathbb{S}^2\\
    \end{array}
\]

\paragraph{ Step 6. Vertex Curvature $K$} The vertex curvature is defined as angle deficit
\[
    K(v_i) = \left\{
    \begin{array}{rl}
    2\pi - \sum_{[v_i,v_j,v_k]} \theta_i^{jk} & v_i \not\in \partial \Sigma\\
    \pi - \sum_{[v_i,v_j,v_k]} \theta_i^{jk} & v_i \not\in \partial \Sigma\\
    \end{array}
    \right.
\]

\begin{table}
{\tiny
\begin{tabular}{|l|l|l|l|l|}
\hline
&$u_i$ & Edge Length $l_{ij}$ & $\tau(i,j,k)$ & $s(x)$ \\
\hline
\hline
$\mathbb{E}^2$ & $\log\gamma_i$ & $l_{ij}^2 = 2\eta_{ij} e^{u_i+u_j} + \epsilon_i e^{2u_i}+ \epsilon_j e^{2u_j}$ & $\frac{1}{2}(l_i^2+ \epsilon_j\gamma_j^2-\epsilon_k\gamma_k^2) $ & $x$\\
\hline
$\mathbb{H}^2$ & $\log\tanh\frac{\gamma_i}{2}$ & $\cosh l_{ij} = \frac{4\eta_{ij}+(1+\epsilon_i e^{2u_i})(1+\epsilon_j e^{2u_j})}{(1-\epsilon_i e^{2u_i})(1-\epsilon_j e^{2u_j})}$ & $\cosh l_i \cosh^{\epsilon_j} \gamma_j - \cosh^{\epsilon_k} \gamma_k $ &
 $\sinh x$\\
\hline
$\mathbb{S}^2$ &  $\log\tan \frac{\gamma_i}{2}$ & $\cos l_{ij} = \frac{4\eta_{ij}+(1-\epsilon_i e^{2u_i})(1-\epsilon_j e^{2u_j})}{(1+\epsilon_i e^{2u_i})(1+\epsilon_j e^{2u_j})}$ & $\cos l_i \cos^{\epsilon_j} \gamma_j - \cos^{\epsilon_k} \gamma_k $ &
 $\sin x$ \\
\hline
\end{tabular}}
\caption{Formulae for $\mathbb{E}^2$, $\mathbb{H}^2$ and $\mathbb{S}^2$ background geometries.\label{tab:summary}}
\vspace{-0.5cm}
\end{table}

\normalsize

\paragraph{ Step 7. Hessian Matrix $H$}
\[
    \frac{\partial(\theta_i,\theta_j,\theta_k)}{\partial(u_i,u_j,u_k)} = -\frac{1}{2A} L \Theta L^{-1} D,
\]
where
\[
    A = \sin\theta_i s(l_j) s(l_k),
\]
and
\[
    L = diag( s(l_i), s(l_j), s(l_k) ),
\]
%
and
\[
D = \left(
\begin{array}{ccc}
0 & \tau(i,j,k) & \tau( i,k,j )\\
\tau(j,i,k) & 0 & \tau( j,k,i )\\
\tau(k,i,j) & \tau(k,j,i) & 0
\end{array}
\right).
\]

\begin{figure*}
\begin{center}
\begin{tabular}{ccc}
Euclidean $\mathbb{E}^2$ & Hyperbolic $\mathbb{H}^2$ & Spherical $\mathbb{S}^2$ \\
\includegraphics[width=0.31\textwidth]{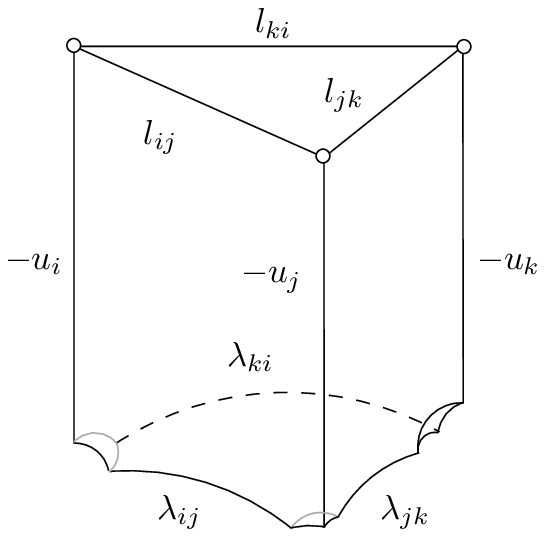}&
\includegraphics[width=0.31\textwidth]{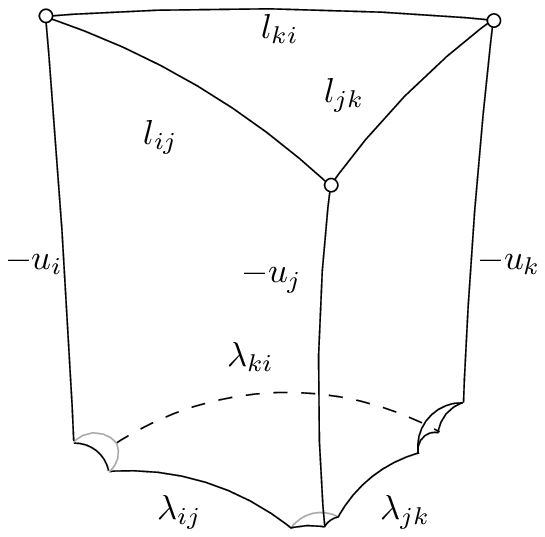}&
\includegraphics[width=0.31\textwidth]{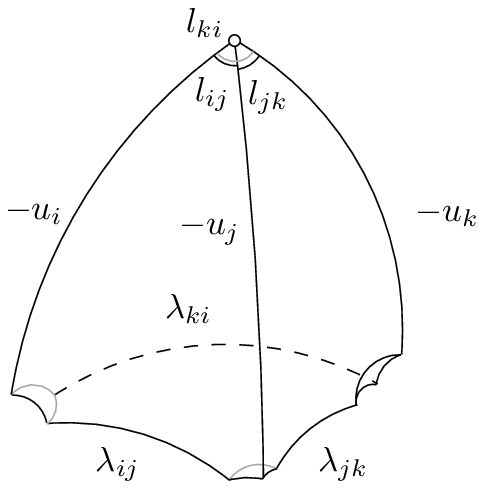}\\
&Inversive Distance Circle Packing & \\
&&\\
\includegraphics[width=0.31\textwidth]{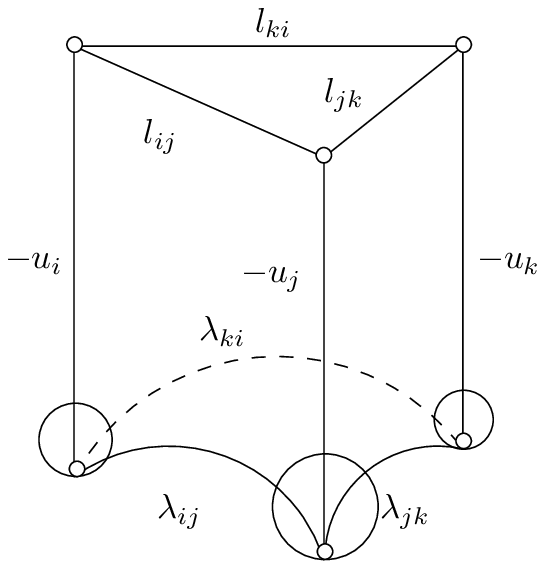}&
\includegraphics[width=0.31\textwidth]{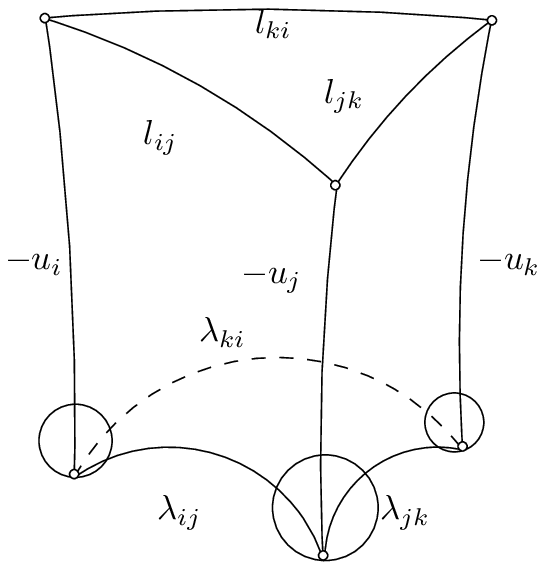}&
\includegraphics[width=0.31\textwidth]{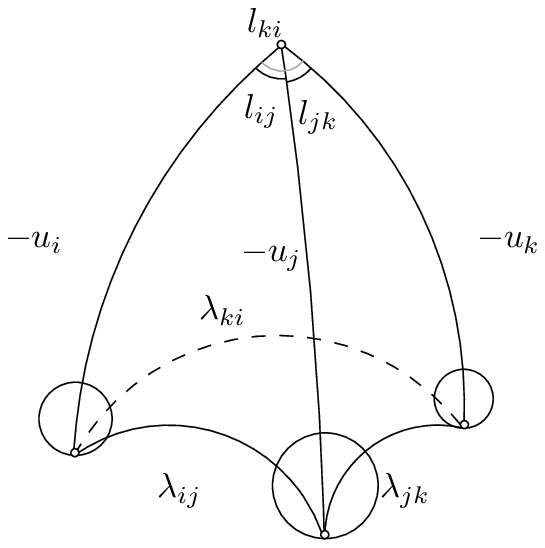}\\
&Yamabe flow & \\
&&\\
\includegraphics[width=0.31\textwidth]{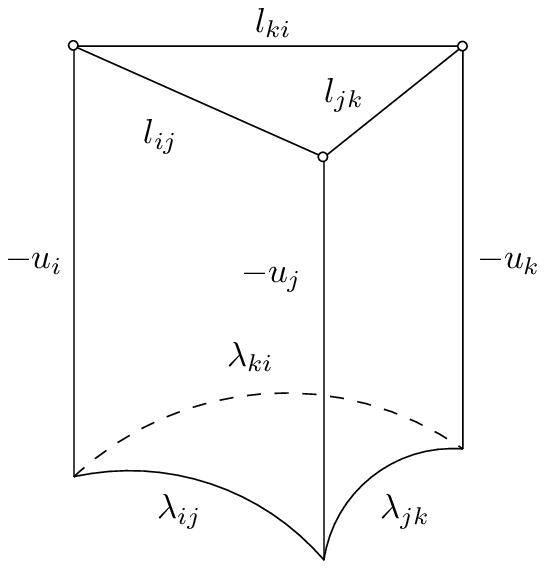}&
\includegraphics[width=0.31\textwidth]{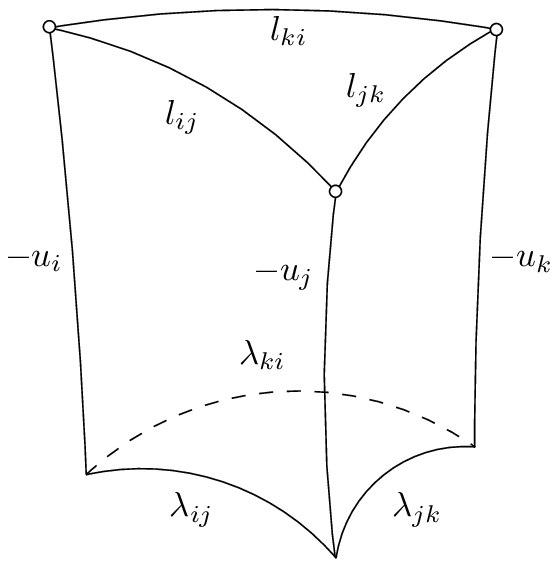}&
\includegraphics[width=0.31\textwidth]{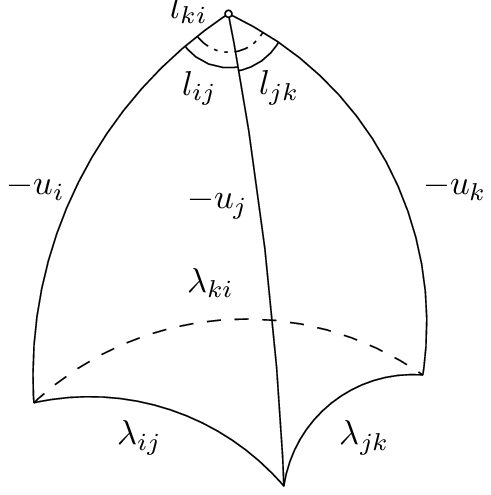}\\
&Virtual radius Circle Packing &\\
&&\\
\includegraphics[width=0.31\textwidth]{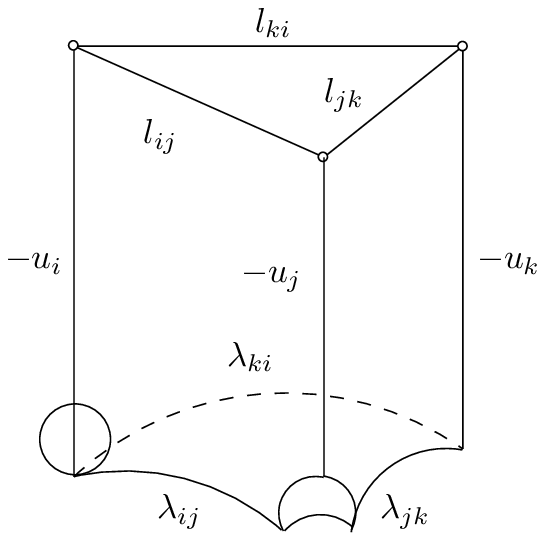}&
\includegraphics[width=0.31\textwidth]{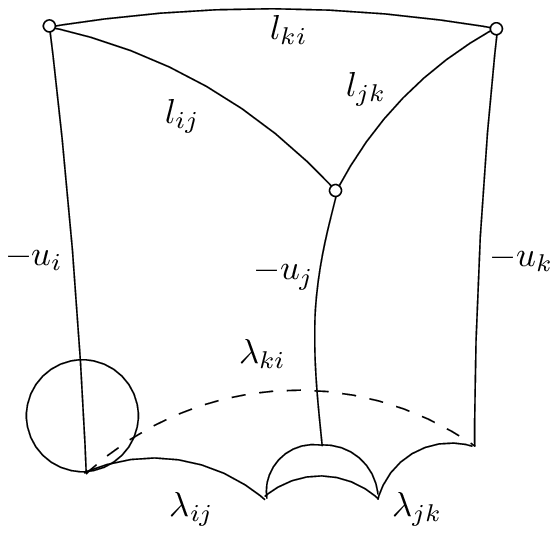}&
\includegraphics[width=0.31\textwidth]{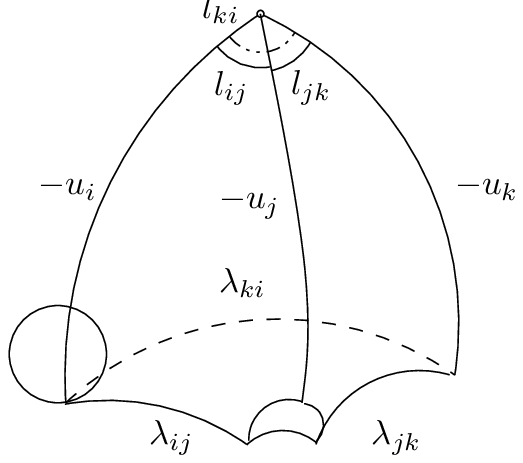}\\
&Mixed type schemes &\\
\end{tabular}
\caption{Geometric interpretation to discrete Ricci energy - volumes
of generalized hyperbolic tetrahedra.\label{fig:energy_volume}}
\end{center}
\end{figure*}
\paragraph{Step. 8 Linear System} If the $\Sigma$ is with $\mathbb{H}^2$ background geometry, then the Hessian matrix $H$ is positive define; else if $\Sigma$ is with $\mathbb{E}^2$ background geometry, then $H$ is positive definite on the linear subspace $\sum_i u_i=0$.
The linear system can be solved using any sparse linear solver, such as Eigen \cite{eigenweb}.

For discrete surface Ricci flow with topological surgeries, we can
add one more step right after step 4. In this new step, we modify
the connectivity of $\Sigma$ to keep the triangulation to be (Power)
Delaunay. This will greatly improves the robustness as proved in \cite{Gu_Luo_Sun_Wu_13} and \cite{Gu_Guo_Luo_Sun_Wu_14}.


\end{document}